\documentclass[11pt]{article}
\usepackage{amstext, amsmath,latexsym,amsbsy,amssymb,amsmath}

\usepackage{enumitem}

\newtheorem{proposition}{Proposition}[section]
\newtheorem{remark}{Remark}[section]
\newenvironment{proof}{{\bf Proof\ }}{\QED\\}
\newtheorem{lemma}{Lemma}[section]

\numberwithin{equation}{section}
\newtheorem{theorem}{Theorem}[section]
\newtheorem{corollary}{Corollary}[section]

\newcommand{\QED}{\hspace*{\fill}\rule{2.5mm}{2.5mm}}

\usepackage{amssymb}\usepackage{graphicx}
\usepackage{tikz}

\usepackage{color}
\newcommand\qed{\hfill$\sqcap\kern-7.5pt\hbox{$\sqcup$}$}

\newcommand{\RR}{\mathbb{R}}

\newcommand{\beqn}{\begin{equation}}
\newcommand{\eeqn}{\end{equation}}
\newcommand{\bear}{\begin{eqnarray}}
\newcommand{\eear}{\end{eqnarray}}
\newcommand{\bean}{\begin{eqnarray*}}
\newcommand{\eean}{\end{eqnarray*}}

\newcommand{\cE}{\mathcal{E}}

\newcommand{\eps}{\varepsilon}

\allowdisplaybreaks

\begin{document}
\title{On the kinetic equation in Zakharov's wave turbulence theory
for capillary waves}

\author{ Toan T. Nguyen\footnotemark[1] \and Minh-Binh Tran\footnotemark[2]
}

\footnotetext[1]{Department of Mathematics, Pennsylvania State University, State College, PA 16802, USA. \\Email: nguyen@math.psu.edu. 
}

\footnotetext[2]{Department of Mathematics, University of Wisconsin-Madison, Madison, WI 53706, USA. \\Email: mtran23@wisc.edu
}

\maketitle
\begin{abstract} 

The wave turbulence equation is an effective kinetic equation that describes the dynamics of wave spectrum in weakly nonlinear and dispersive media. Such a kinetic model has been derived by physicists in the sixties, though the well-posedness theory remains open, due to the complexity of resonant interaction kernels. In this paper, we provide a global unique radial strong solution, the first such a result, to the wave turbulence equation for capillary waves. 


 \end{abstract}



\tableofcontents

\section{Introduction}


Over the last 60 years the theory of weak turbulence has been intensively developed. In weakly nonlinear and dispersive wave models, the weak turbulence kinetic equation can be formally derived, via the statistical approach, to describe the dynamics of resonant wave interactions. The model for slightly viscous capillary waves on the surface
of a liquid reads as follows  (cf. \cite{pushkarev1996turbulence, pushkarev2000turbulence,zakharov1968stability,zakharov1967weak})
\begin{equation}\label{WeakTurbulenceInitial}
\begin{aligned}
\partial_tf + 2  \nu |k|^2 f \ =& \ Q[f]
\end{aligned}
\end{equation}
in which  $f(t,k)$ is the nonnegative wave density  at  wavenumber $k \in \RR^d$, $d \ge 2$. Here, $\nu$ denotes the positive coefficient of fluid viscosity (strictly speaking, the model is derived under the assumption $\nu \sqrt k\ll 1$ so that the dispersion remains dominating the viscous dissipation; see \cite{viscousWW} for more details on the addition of the viscous damping). The term $Q[f]$ denotes the integral collision operator, describing pure resonant
three-wave interactions. The equation is a  three-wave kinetic one, in which the collision operator is of the form 
\begin{equation}\label{def-Qf}Q[f](k) \ = \ \iint_{\mathbb{R}^{2d}} \Big[ R_{k,k_1,k_2}[f] - R_{k_1,k,k_2}[f] - R_{k_2,k,k_1}[f] \Big] dk_1dk_2 \end{equation}
with $$\begin{aligned}
R_{k,k_1,k_2} [f]:=  4\pi |V_{k,k_1,k_2}|^2\delta(k-k_1-k_2)\delta(\mathcal{E}_k -\mathcal{E}_{k_1}-\mathcal{E}_{k_2})(f_1f_2-ff_1-ff_2) 
\end{aligned}
$$
with the short-hand notation $f = f(t,k)$ and $f_j = f(t,k_j)$. The Dirac delta function $\delta(\cdot)$ is to ensure the following resonant conditions for the wavenumbers:
\begin{equation}\label{cv} k = k_1 + k_2 , \qquad \cE_k = \cE_{k_1} + \cE_{k_2} ,\end{equation}
with $\cE_k$ denoting the dispersion relation of the waves. The exact form of the collision kernel $V_{k,k_1,k_2}$ will be recalled below. 

According to the weak turbulence theory (cf. \cite{zakharov1967weak,zakharov2012kolmogorov,KorotkevichDyachenkoZakharov:2016:NSO}), equation \eqref{WeakTurbulenceInitial} in the absence of viscosity admits nontrivial equilibria $f_\infty$, called the Kolmogorov-Zakharov's spectra:  $$f_\infty(k) \approx C|k|^{-\frac{17}{4}}.$$
Moreover, such a solution can be interpreted as a universal spectrum in the region of transparency.
 These solutions are the analogs of the familiar Kolmogorov energy spectrum prediction
$C|k|^{-\frac{5}{3}}$ of hydrodynamic turbulence.

The derivation of the equation  dated back to the 60's, starting with the pioneering work of Hasselmann, Zakharov and collaborators (cf. \cite{hasselmann1962non,hasselmann1974spectral,zakharov1965weak,zakharov1968stability,zakharov1967weak}). Since then, a lot works have been done,  trying to understand the equation (see 
 \cite{zakharov1965weak,KorotkevichDyachenkoZakharov:2016:NSO,gamba2017wave,GermainIonescuTran,zakharov1999statistical,zakharov1968stability,zakharov2012kolmogorov,
 connaughton2009numerical,zakharov1967weak,newell2011wave,balk1990physical,
 pushkarev1996turbulence,pushkarev2000turbulence,faou2016weakly,buckmaster2016effective,germain2015continuous,buckmaster2016analysis,germain2016continuous,EscobedoVelazquez:2015:OTT,smith2002generation,Merino:Thesis:2015,l1997statistical} and references therein). We refer to the books \cite{Nazarenko:2011:WT} for more discussions and references on the topic. Due to its complexity, the fundamental question on the global existence and uniqueness of solutions to the equation is still {\it unsolved}. In this paper, we give, {\it for the first time}, an answer to the  fundamental question on the global existence and uniqueness of solutions to the equation, for the case where the solutions are radial.

In this paper, we develop new techniques, inspired by  recent works on quantum kinetic theory. Let us mention that the  kinetic wave equation \eqref{WeakTurbulenceInitial} has a very similar structure with the quantum Boltzmann equation that describes the evolution of the excitations in a trapped Bose gas system, in which the temperature of the gas is below the Bose-Einstein condensate transition temperature (cf. \cite{PomeauBrachetMetensRica,ZakharovNazarenko:DOT:2005,MR1837939,josserand2001nonlinear,QK0,KD1}). The collision operator that describes the interaction between excitations and condensates in the quantum Boltzmann equation reads
\begin{equation}\label{QBCollision1}C[f](k) \ = \ \iint_{\mathbb{R}^{2d}} \Big[ \bar{R}_{k,k_1,k_2}[f] - \bar{R}_{k_1,k,k_2}[f] - \bar{R}_{k_2,k,k_1}[f] \Big] dk_1dk_2 \end{equation}
with \begin{equation}\label{QBCollision2}\begin{aligned}
\bar{R}_{k,k_1,k_2} [f]:=  |\bar{V}_{k,k_1,k_2}|^2\delta(k-k_1-k_2)\delta(\bar{\mathcal{E}}_k -\bar{\mathcal{E}}_{k_1}-\bar{\mathcal{E}}_{k_2})(f_1f_2-ff_1-ff_2-f) 
\end{aligned}
\end{equation}
and $|\bar{V}_{k,k_1,k_2}|^2=\mathcal{C}^*|k||k_1||k_2|$, $\bar{\mathcal{E}}_k=\sqrt{\kappa_1 |k|^2 + \kappa_2 |k|^4}$ for some positive constants $\mathcal{C}^*,$ $\kappa_1,$ $\kappa_2$.  Recent progresses on  quantum Boltzmann equations (cf. \cite{Binh9,ToanBinh,SofferBinh1,SofferBinh2,AlonsoGambaBinh,JinBinh,reichl2017kinetic}) have opened some opportunities to tackle this  open problem, the existence and uniqueness of solutions to \eqref{WeakTurbulenceInitial}.

We note that, in the absence of the linear term in \eqref{QBCollision2} or the viscous damping in \eqref{WeakTurbulenceInitial}, singularities are likely to form. Indeed, \cite{Spohn:2010:KOT} constructed a self-similar blowup solution to the quantum Boltzmann equation, when the linear term is dropped.


\subsection{Main result}

Throughout the paper, we consider the following generalized version of \eqref{WeakTurbulenceInitial}

\begin{equation}\label{WeakTurbulenceGeneralized}
\begin{aligned}
\partial_tf + 2  \nu (|k|^2+\varrho |k|^4) f \ =& \ Q[f]
\end{aligned}
\end{equation}
for $\varrho\geq 0$. 
Solutions to the original model \eqref{WeakTurbulenceInitial} will be obtained via the limit of $\varrho \to 0$. 




The law of wave dispersion on the surface of infinitely deep liquid is of the form\begin{equation}\label{def-Ek} \cE_k  = \sqrt{\sigma |k|^3}\end{equation}
for $\sigma$ the surface tension coefficient, and the collision kernel $V_{k,k_1,k_2}$ is defined by 
\begin{equation}\label{def-VV}
\begin{aligned}
V_{k,k_1,k_2}  \ = \frac{1}{8\pi \sqrt{2\sigma}} \sqrt{\cE_k \cE_{k_1} \cE_{k_2}} \Big( \frac{L_{k_1,k_2}}{ |k| \sqrt{|k_1| |k_2|}} - \frac{L_{k,-k_1}}{ |k_2| \sqrt{|k| |k_1|}} - \frac{L_{k,-k_2}}{ |k_1| \sqrt{|k| |k_2|}} \Big) 
\end{aligned}
\end{equation}
with $L_{k_1,k_2} = k_1 \cdot k_2 +|k_1| |k_2|$; see \cite{pushkarev1996turbulence, pushkarev2000turbulence}. 
Without loss of generality, we assume the surface tension $\sigma=1$. In the scope of our paper, we only consider the case $d=2$ or $3$, which are relevant dimensions in the physical applications. 

We shall construct global unique radial solutions to \eqref{WeakTurbulenceGeneralized} in weighted $L^1$ spaces. Precisely, for $N>0$, let $L^1_N(\RR^d)$ be the function space consisting of $f(k)$ so that the norm 
$$ \| f\|_{L^1_N} : = \int_{\RR^d} f(k) \cE_k^N \; dk $$
is finite, with the dispersion relation $\cE_k$ defined as in \eqref{def-Ek}. In addition, for any $N>0$ and $\vartheta_0>0$, we introduce 
$$\mathcal{S}_N: = \Big\{ f\in L_{\frac13}^1(\RR^d) \cap L_{N+3}^1(\RR^d) ~:~ f\ge 0, \quad f(k) = f(|k|), \qquad \| f\|_{L^1_1} + \| f\|_{L^1_{N+3} } \le \vartheta_0\Big\} .$$
We shall construct solutions of \eqref{WeakTurbulenceGeneralized} that remain $\mathcal{S}_N$, if initially so. 

Our main result is as follows. 

\begin{theorem}\label{Theorem:Main}
Let $N>0$, and let $f_0(k)=f_0(|k|) \in \mathcal{S}_N \cap L^2(\RR^d)$, $d=2$ or $3$. Then for all $\varrho\geq 0$, the weak turbulence equation \eqref{WeakTurbulenceGeneralized}, with initial data $f(0,k) = f_0(k)$ and $\nu>0$, has a unique global solution $f(t,k)$ so that
\begin{equation}\label{the_theorem}
0\leq f(t,k)=f(t,|k|)\in \mathcal{C}\big([0,\infty); \mathcal{S}_N\cap L^2(\RR^d)\big)\cap  \mathcal{C}^1\big((0,\infty);L^1_N\big).
\end{equation}
Moreover, there holds the propagation of moments: for any $n\ge 1$, if $f_0 \in L^1_n$, then there exists $C_{n}>0$ such that 
\begin{equation}
\sup_{t\geq 0 }\|f(t,\cdot)\|_{L^1_{n}}\le C_{n}.
\end{equation}
\end{theorem}
\begin{remark} Notice that for the case where $\varrho>0$, we have a stronger result: we can remove the $L^2$ dependence on the initial condition, and the solution exists in $\mathcal{C}\big([0,\infty); \mathcal{S}_N\cap L^1(\RR^d)\big)\cap  \mathcal{C}^1\big((0,\infty);L^1_N\big)$. 
\end{remark}

Let us mention that the classical Boltzmann equation describes the evolution of the density function of a dilute classical gas. After the collision, two particles with velocities $k_1$ and $k_2$ change their velocities into $k_3$ and $k_4$. Since   the energy of the particles is of the form  $\cE_k = |k|^2$; the  conservation of moment and energy then read 
$$|k_1|^2+|k_2|^2=|k_3|^2+|k_4|^2, \ \ \ k_1+k_2=k_3+k_4. $$
As a consequence, $k_1$, $k_2$, $k_3$, $k_4$ belong to the sphere centered at $\frac{k_1+k_2}{2}$ with radius $\frac{|k_1-k_2|}{2}$ and the classical Boltzmann collision operator can be expressed as a integration on a sphere (cf. \cite{Carleman:1933:TEI,Villani:2002:RMT}).

 Let us now turn to the collision operator of \eqref{WeakTurbulenceGeneralized}. As in the classical case, the collision operator involves surface integrals. Precisely, we introduce functions
 \begin{equation}\label{FunctionH}\begin{aligned}
H_0^k(w): = \mathcal{E}_{k-w} + \mathcal{E}_w - \mathcal{E}_k, \ \ \ H_1^k(w): =\mathcal{E}_k + \mathcal{E}_w - \mathcal{E}_{k+w},
\end{aligned}\end{equation}
and  the energy surfaces, dictated by the resonant conditions \eqref{cv}, 
\begin{equation}\label{def-Sp}\begin{aligned}
S_k:& = \Big \{ w\in \mathbb{R}^d~:~H^0_k(w)=0 \Big\}
\\
S'_k : &= \Big \{w\in \mathbb{R}^d~:~ H^1_k(w)= 0 \Big\}
\end{aligned}\end{equation}
with $\cE_k = \sqrt\sigma |k|^{3/2}$. 
The collision operator $Q[f]$ then reduces to 
\begin{equation}\label{WT-Q} Q[ f](k)= \int_{S_k} R_{k,k - k_2,k_2}[ f]  \frac{d\sigma(k_2)}{|\nabla H_0^k(k_2)|} - 2 \int_{S'_k} R_{k+k_2,k,k_2} [ f]  \frac{d\sigma(k_2)}{|\nabla H_1^k(k_2)|}.\end{equation}
Difficulties arise. First, surfaces $S_k$ and $S_k'$ are no longer a sphere as in the classical case, and the analysis on these surfaces can be tricky. More seriously, due to the lack of an integration over the whole space (compare with the classical Boltzmann equation), we are forced to bound surface integrals in term of (weighted) $L^1$ norms of solutions, a type of estimates that are in general false. In Section \ref{Sec:Energy}, we shall derive such estimates for radial functions.  

By view of \eqref{cv}, the weak turbulence equation \eqref{WeakTurbulenceGeneralized} conserves momentum and energy (in the absence of viscous damping), but does not conserve mass. As a consequence, one of the issues in dealing with \eqref{WeakTurbulenceGeneralized} is that $L^1$ norms, say with weight $|k|^n$, of solutions do not close by itself, but are bounded by $L^1$ norms with a much higher-order weight. This is due to the high-order collision kernel. That is, roughly speaking, the kernel $|V_{k,k_1,k_2}|^2$ is of order $9/2$ in $|k|$, which is much higher than the order of classical Boltzmann collision kernel (typically, smaller than one). This apparent loss of weights in $|k|$ gives the impression that solutions could blow up in finite time, even in the presence of viscous damping: $2\nu |k|^2f$, which gains precisely two order in $|k|$.




%
%

We end the introduction by giving the structure of the paper: 

\begin{itemize}
\item We derive the momentum and energy identities and provide a careful study of the surface integrals on the energy manifolds. 

\item In Section \ref{Sec:MomentPropa}, we provide an a priori estimate on the $L^1_N$ norm of the solution. 

\item An $L^2$ estimate on the solutions of \eqref{WeakTurbulenceGeneralized} and the H\"older continuity of the collision operator will be established in Sections \ref{Sec:L2} and \ref{Sec:HolderEstimate}, respectively.

\item The proof of Theorem \ref{Theorem:Main} is given in Section \ref{Sec:Main}.

\end{itemize}

\section{Conservation laws and energy surfaces}\label{Sec:Scaling}

%
%

\subsection{Momentum and energy identity}\label{Sec:MomentEnergy}In this section, we obtain the basic properties of strong solutions of \eqref{WeakTurbulenceGeneralized}.  
\begin{lemma}\label{Lemma:WeakFormulation}
There holds 
$$
\begin{aligned}
\int_{\RR^d}Q[f](t,k) \varphi(k) \; dk  \ = &\  \iiint_{\mathbb{R}^{3d}} R_{k,k_1,k_2} [f]  \Big[ \varphi(k) - \varphi(k_1) - \varphi(k_2)\Big] dk dk_1dk_2 
\end{aligned}
$$
for any test functions $\varphi$ so that the integrals make sense. 
\end{lemma}
\begin{proof} By definition, we compute 
$$
\begin{aligned}
\int_{\RR^d}Q[f](t,k) \varphi(k) \; dk  \ = &\ \iiint_{\mathbb{R}^{3d}} \Big[ R_{k,k_1,k_2} - R_{k_1,k,k_2} - R_{k_2,k,k_1} \Big] \varphi(k) dk dk_1dk_2 .
\end{aligned}
$$
By switching the variables $k\leftrightarrow k_1$, $k\leftrightarrow k_2$ in the first integral on the right, the lemma follows at once.  
\end{proof}

As a direct consequence, we obtain the following corollary. 

\begin{corollary}[Momentum and energy identities] Smooth solutions $f(t,k)$ of \eqref{WeakTurbulenceGeneralized} satisfy 
\begin{equation}\label{Coro:ConservatioMomentum}
\frac{d}{dt}\int_{\mathbb{R}^d}f(t,k)kdk + 2\nu  \int_{\mathbb{R}^d}f(t,k)k(|k|^2+\varrho |k|^4) dk=0
\end{equation}
and
\begin{equation}\label{Coro:ConservatioEnergy}
\frac{d}{dt}\int_{\mathbb{R}^d}f(t,k)\mathcal{E}_kdk + 2\nu  \int_{\mathbb{R}^d}f(t,k)\mathcal{E}_k(|k|^2+\varrho |k|^4)dk=0
\end{equation}
for all $t\ge 0$. 
\end{corollary}
\begin{proof} This follows from Lemma \ref{Lemma:WeakFormulation} by taking $\varphi(k) = k$ and $\varphi(k) = \mathcal{E}_k$, and using the resonant conditions \eqref{cv}.
 \end{proof}


%
%
%

\subsection{Energy surfaces}\label{Sec:Energy}
Our first step is to study the surface integrals. For sake of generality, we consider in this section the following power-law energy function 
\begin{equation}\label{def-Epp}\cE_k = \cE(k) = |k|^\gamma, \qquad 1 < \gamma \le 2.\end{equation}
In the case when $\gamma=1$, the surface $S_k$ degenerates into a straight line $S_k = \{\alpha k\}_{\alpha \in [0,1]}$, and the surface integral reduces to a line integral. Such an energy corresponds to the  dispersion law of phonons, and has been studied in \cite{AlonsoGambaBinh,CraciunBinh,Binh9}.

\begin{figure}[t]
\centering
\includegraphics[scale=.45]{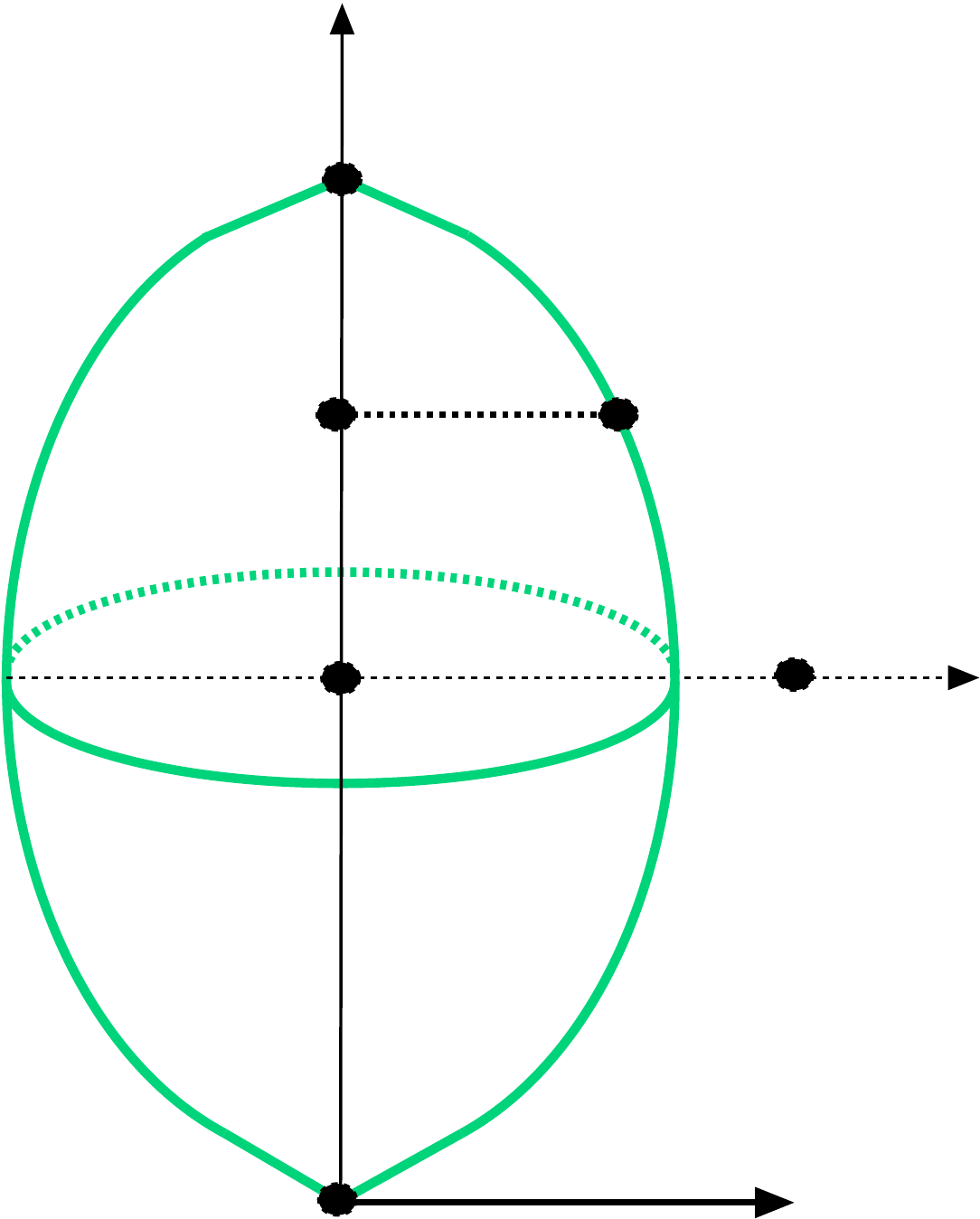}
\put(-95,-10){$0$}
\put(-87,155){$p$}
\put(-25,85){$\frac{|p|}2$}
\put(-30,10){$q$}
\put(-87,85){$\frac{p}2$}
\put(-110,117){$\alpha p$}
\put(-50,120){$w_\alpha$}
\put(-85,120){$s(\alpha)$}
\caption{\em Sketched is the surface $S_p$, centered at $\frac p2$ and having $0$ and $p$ as its south and north poles, respectively.}
\label{fig-Sp}
\end{figure}
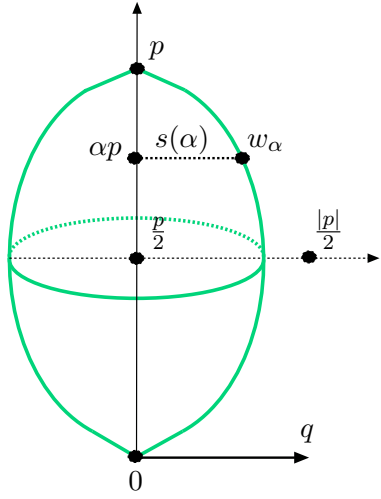

\begin{lemma}[Surface $S_p$]\label{lem-Sp} Let $\gamma \in (1,2]$ and $S_p$ be defined as in \eqref{def-Sp}-\eqref{def-Epp}. Then, for each $p$, $S_p = |p|S_{e_p}$, with $e_p = p/|p|$. 
In addition, there hold the following properties:

i. $\{0,p\}\subset S_p \subset \overline{B(\frac p2,\frac{|p|}{2})}$. 

ii. The surface $S_p$ is invariant under the rotation around $p$, and can be parametrized by 
$$ S_p = \Big\{ w (\alpha,e_q) = \alpha p + s(\alpha) e_q ~:~\alpha \in [0,1], \quad |e_q|=1, \quad e_q \cdot p =0\Big\} $$
for some function $s(\alpha)$ that is smooth in $(0,1)$; see Figure \ref{fig-Sp}.  In the two dimensional case: $d=2$, $S_p$ is a curve parametrized by $\alpha\in [0,1]$. 

iii. $s(\alpha) = s(1-\alpha)$, $s(0)=0$, and $s(\alpha)$ is strictly increasing and invertible on $(0,\frac12)$.

iv. There are universal constants $c_0,C_0$ so that the surface area satisfies 
$$  c_0 |p|^{d-\gamma}  \le  \int_{S_p} \frac{d\sigma(w)}{|\nabla_w H_0^p(w)|} \le  C_0 |p|^{d-\gamma} ,$$
uniformly in $p\in \mathbb{R}^d$. 

v. There holds
$$ \int_{S_p} F(|w|) \frac{d\sigma(w)}{|\nabla_w H_0^p(w)|} \le C|p|^{d-\gamma-2} \int_{|w|\le |p|} F(|w|) |w|^{2-d} dw.$$
\end{lemma}
\begin{proof} Let $p\in \RR^d \setminus \{0\}$. It is clear that $S_p = |p| S_{e_p}$, with $e_p = p/|p|$. Thus, it suffices to study the case when $|p|=1$. As for (i), it is clear that $\{0,p\}\subset S_p$.  Next, since $a^\gamma + b^\gamma \le (a+b)^\gamma$, we have for $w\in S_p$
$$ 1 = \Big( |p-w|^\gamma + |w|^\gamma\Big)^{2/\gamma} \ge |p-w|^2 + |w|^2.$$
This proves that $w \in \overline{B(\frac p2,\frac{1}{2})}$. (i) follows.

As for (ii), we write $w = \alpha p + q$ for $q$ orthogonal to $p$. By orthogonality, $|w|$ and $|w-p|$ do not depend on the direction of $q$, and neither does $S_p$. That is, $S_p$ is invariant under the rotation around $p$. We set 
$$w(\alpha,s) = \alpha p + s e_q.$$
We shall prove the existence of a function $s=s(\alpha)$ for $\alpha\in (0,1)$, so that $w(\alpha,s(\alpha))\in S_p$. To this end, let 
$$H_0^p(\alpha,s): = \mathcal{E}(p-w(\alpha,s))  + \mathcal{E}(w(\alpha,s)) - \mathcal{E}(p),$$
as in \eqref{FunctionH}.  
Clearly, $H_0^p(\alpha,s) =0$ if and only if $w(\alpha,s) \in S_p$. Observe that $H_0^p(\alpha,0) <0$ (by convexity of $\cE(p)$) and $H_0^p(\alpha,s) >0$ for sufficiently large $s$ (and hence large $|w(\alpha,s)|)$. The existence of a such $s(\alpha)$ follows. In addition, a direct computation yields 
\begin{equation}\label{DG}
\begin{aligned}
\nabla_w H_0^p
 &= \frac{ w - p}{|p-w|} \mathcal{E}'(p -w) + \frac{w}{|w|} \mathcal{E}'(w)
 \end{aligned}\end{equation}
and hence $\partial_s H_0^p(\alpha,s) = e_q \cdot \nabla_w H_0^p$ is positive, since $\cE'(w) >0$ (for $w\not=0$). That is, $H_0^p(\alpha,s)$ is increasing in $s$ for each $\alpha$ and $s(\alpha)$ is uniquely determined, so that $H_0^p(\alpha,s(\alpha)) =0$. The smoothness of $s(\alpha)$ follows from that of $\cE(\cdot)$. This proves (ii).

Next, the symmetry stated in (iii) is clear from the definition of $S_p$, and it suffices to study $s(\alpha)$ for $\alpha\in (0,\frac12]$. Observe that for $w_\alpha = w(\alpha,s(\alpha)) \in S_p$, we have $0=F(\alpha,s(\alpha))$ and 
\begin{equation}\label{id-dF}
\begin{aligned}
 0 &= \partial_\alpha H_0^p + s'(\alpha) \partial_s H_0^p = p \cdot \nabla_w H_0^p + s'(\alpha) e_q \cdot \nabla_w H_0^p
 \\
 &= 
   -  \frac{\mathcal{E}'(p -w_\alpha)}{|p-w_\alpha|} + ( s(\alpha) s'(\alpha)  + \alpha )\Big[ \frac{\mathcal{E}'(p -w_\alpha)}{|p-w_\alpha|} + \frac{\mathcal{E}'(w_\alpha) }{|w_\alpha|}\Big] .
 \end{aligned}\end{equation}
Setting $\cE_1(w):= |w|^{\gamma-2}$, we have 
\begin{equation}\label{ss-pr}s(\alpha) s'(\alpha) =\frac{ (1-\alpha) \cE_1(p-w_\alpha) - \alpha \cE_1(w_\alpha)  }{ \cE_1(p-w_\alpha) + \cE_1(w_\alpha)} .\end{equation}
Observe that the function in the numerator in \eqref{ss-pr} is decreasing in $\alpha$, and vanishes at $\alpha =\frac12$. Hence, for $\alpha \in (0,\frac12)$, we have $s'(\alpha) >0$, and hence $s(\alpha)$ is invertible, on $(0,\frac12)$. This yields (iii).

Next, we compute the surface area of $S_p$. Let us consider the case when $d\ge 3$; the case when $d=2$ is simpler, as the surface is parametrized solely by $\alpha \in [0,1]$. Writing $w(\alpha,\theta) = \alpha p + s(\alpha) e_\theta$, we have 
\begin{equation}\label{dS}\begin{aligned}
 d \sigma (w) &= |\partial_\alpha w \times \partial_\theta w | d\alpha d\theta  = \Big |(p+s'(\alpha) e_\theta) \times s(\alpha)\partial_\theta e_\theta \Big | d\alpha d\theta
\\
 &=
 \sqrt{|p|^2  |s(\alpha) | ^2+ \frac14| \partial_\alpha  (|s(\alpha) |^2)|^2}d\alpha  d\theta.
 \end{aligned}\end{equation}
 
 We deduce from \eqref{id-dF} that
 \begin{equation}\label{dG-qa}
\begin{aligned}
 0 &= \partial_\alpha w_\alpha \cdot \nabla_w H_0^p\\
 & =
  \frac12 \partial_\alpha |s(\alpha)|^2  \left[ \frac{\mathcal{E}'(|p -w_\alpha|)}{|p-w_\alpha|} + \frac{\mathcal{E}'(|w_\alpha|) }{|w_\alpha|}\right]  + \alpha |p|^2 \frac{\mathcal{E}'(|w_\alpha|) }{|w_\alpha|} - (1-\alpha) |p|^2 \frac{\mathcal{E}'(|p -w_\alpha|)}{|p-w_\alpha|}.
 \end{aligned}\end{equation}
 It is straightforward that
 \begin{equation}\label{proposition:T2L2:E8}
\begin{aligned}
  \partial_\alpha |s(\alpha)|^2  &
 &= 
2|p|^2 \frac{\gamma   \frac{\mathcal{E}'(|w_\alpha|) }{|w_\alpha|} +(\alpha-1)  \frac{\mathcal{E}'(|p -w_\alpha|)}{|p-w_\alpha|}}{\frac{\mathcal{E}'(|p -w_\alpha|)}{|p-w_\alpha|} + \frac{\mathcal{E}'(w_\alpha) }{|w_\alpha|}}.
 \end{aligned}\end{equation}

Now, let us compute $|\nabla H_0^p|$ under the new parametrization
\begin{equation*}
\begin{aligned}
 |\nabla H_0^p|^2 \ =  &\  |p|^2\left[\alpha   \frac{\mathcal{E}'(|w_\alpha|) }{|w_\alpha|} +(\alpha-1)  \frac{\mathcal{E}'(|p -w_\alpha|)}{|p-w_\alpha|}\right]^2\\
 & \ + |s(\alpha)|^2 \left[\frac{\mathcal{E}'(|p -w_\alpha|)}{|p-w_\alpha|} + \frac{\mathcal{E}'(|w_\alpha|) }{|w_\alpha|}\right]^2,
\end{aligned}
\end{equation*}
which, in companion with \eqref{proposition:T2L2:E8}, implies
\begin{equation}\label{proposition:T2L2:E9}
\begin{aligned}
 |\nabla  H_0^p|^2 \ =  &\  \frac{\left|\partial_\alpha |s(\alpha)|^2\right|^2}{4|p|^2}\left[\frac{\mathcal{E}'(|p -w_\alpha||)}{|p-w_\alpha||} + \frac{\mathcal{E}'(|w_\alpha||) }{|w_\alpha||}\right]^2\\
 & \ + |s(\alpha)|^2 \left[\frac{\mathcal{E}'(|p -w_\alpha|)}{|p-w_\alpha|} + \frac{\mathcal{E}'(|w_\alpha|) }{|w_\alpha|}\right]^2,
\end{aligned}
\end{equation}
We get the following representation of  $|\nabla H_0^p|$
\begin{equation}\label{proposition:T2L2:E10}
\begin{aligned}
 |\nabla H_0^p| \ =  &\  \frac{\sqrt{ \frac{\left|\partial_\alpha |s(\alpha)|^2\right|^2}{4} +  |\alpha|^2|p|^2}}{|p|}\left[\frac{\mathcal{E}'(|p -w_\alpha|)}{|p-w_\alpha|} + \frac{\mathcal{E}'(|w_\alpha|) }{|w_\alpha|}\right].
\end{aligned}
\end{equation}

As for the surface integral of a radial function $\mathcal{G}(|w|)$, we  introduce the radial variable $ u = |W_\alpha | = \sqrt{\alpha^2 |p|^2 + |q_\alpha|^2}$. We compute $2u du =\partial_\alpha |W_\alpha|^2 d\alpha$ and hence 
$$\frac{d \sigma (w) }{|\nabla H_0^p|}= \frac{|p|}{2\left[\frac{\mathcal{E}'(|p -w_\alpha|)}{|p-w_\alpha|} + \frac{\mathcal{E}'(|w_\alpha|) }{|w_\alpha|}\right]\partial_\alpha |w_\alpha|^2 } u du d\theta,
$$ 
which in combination with \eqref{dG-qa} yields
$$\frac{d \sigma (w) }{|\nabla H_0^p|}= \frac{|p-w_\alpha|}{\mathcal{E}'(|p-w_\alpha|)|p|} du d\theta= \frac{u|p-w_\alpha|^{2-\gamma}}{|p|(\gamma-1)} du d\theta,
$$ 
Since
$$\frac{u||p|+u|^{2-\gamma}}{|p|(\gamma-1)} \le \frac{u|p-w_\alpha|^{2-\gamma}}{|p|(\gamma-1)}\le \frac{u[|p|+u]^{2-\gamma}}{|p|(\gamma-1)},$$
upon noting that $d\sigma(S_p) = |p|^{d-1} d\sigma(S_{e_p})$ and defining $v=\frac{u}{|p|}$, we obtain
$$ \int_{S_p} \frac{d\sigma(w)}{|\nabla_w H_0^p|}  \ge c_1 |p|^{d-\gamma}\int_0^1 v|1-v|^{2-\gamma} dv \ge  c_0|p|^{d-\gamma},$$
and
$$\int_{S_p} \frac{d\sigma(w)}{|\nabla_w H_0^p|} ) \le C_1 |p|^{d-\gamma}\int_0^1v|1+v|^{2-\gamma} dv \le  C_0|p|^{d-\gamma},$$
for some $c_0, c_1,$ $C_0, C_1$, depending only on $\gamma$ (in particular, independent of $p$). This proves (iv). 

Finally, we check the surface integral of a radial function $f(|w|)$. It is clear that $$ \int_{S_p} f(|w|) \frac{d \sigma (w) }{|\nabla H_0^p|} \le C|p|^{d-\gamma-2} \int_0^{|p|} f(u) u du.$$
The lemma follows by the spherical coordinates $dw=|w|^{d-1}d(|w|)d\sigma(\mathbb{S}^d).$ 
\end{proof}

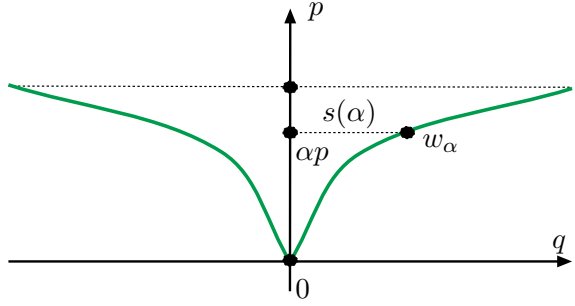
\begin{figure}[t]
\centering
\includegraphics[scale=.45]{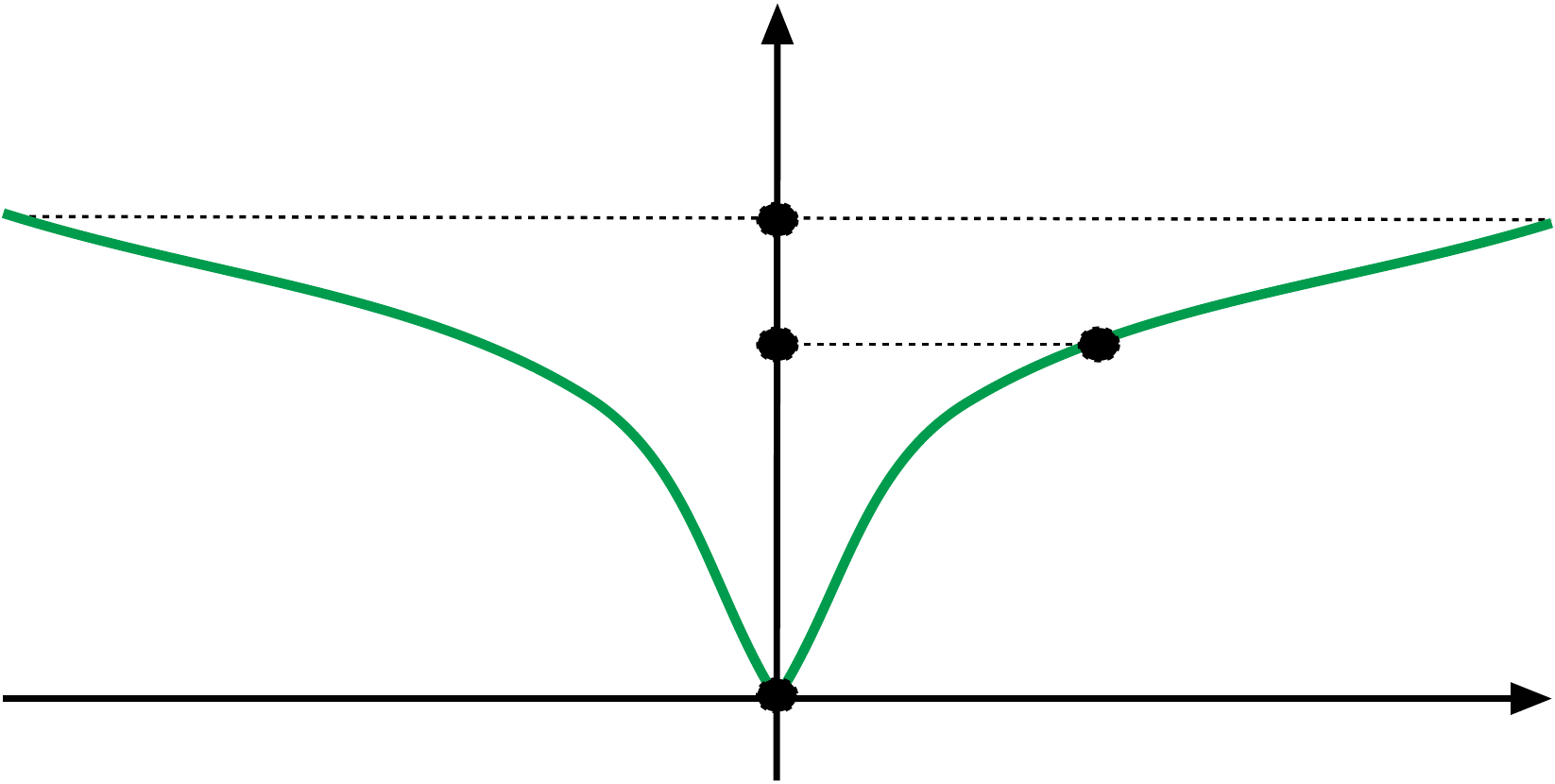}
\put(-105,-2){$0$}
\put(-100,105){$p$}
\put(-8,17){$q$}
\put(-105,50){$\alpha p$}
\put(-57,55){$w_\alpha$}
\put(-95,65){$s(\alpha)$}
\caption{\em Sketched is the trace of $S'_p$ on any two dimensional plane containing $p$.}
\label{fig-Sp1}
\end{figure}

\begin{lemma}[Surface $S_p'$]\label{lem-Sp1} Let $S_p'$ be defined as in \eqref{def-Sp}-\eqref{def-Epp}. Then, $S_p' = |p|S'_{e_p}$, with $e_p = p/|p|$. In addition, 
There is a positive constant $C_0$ so that 
\begin{equation}\label{bd-Sp123} \int_{S_p'} F(|w|)\; \frac{d \sigma(w) }{|\nabla_w H_1^{p}(w)|}\le C_0  |p|^{d-2-\gamma}\int_0^\infty  F(|w|)(1+|w|/|p|)^{2-\gamma}|w|^{2-d} dw\end{equation}
uniformly in $p\in \mathbb{R}^d.$ 
\end{lemma}
\begin{proof} It is clear that $S_p' = |p| S_{e_p}'$, in which the surface 
\begin{equation}S'_{e_p} =\Big \{ w(\alpha,\theta) = \alpha e_p + s(\alpha)e_\theta~:~\alpha \in [0, \alpha_0),~ \theta \in [0,2\pi]\Big\}\end{equation}
for some monotonic function $s(\alpha)$ and some positive constant $\alpha_0$; see Figure \ref{fig-Sp1}. We stress that the parametrization $\alpha, s(\alpha)$ and $\theta$ are independent of $p$. As a consequence, the surface integral on $S'_{e_p}$ is independent of $p$. 

As in \eqref{dS}, we have
$$
 d \sigma (w) = \sqrt{|p|^2  |s(\alpha)| ^2+ \frac14| \partial_\gamma (|s(\alpha)|^2)|^2}d\alpha d\theta 
 $$
and hence, the surface  is estimated by 
$$ 
\begin{aligned}
\frac{d \sigma(w) }{|\nabla H_1^{p}(w)|} \; 
&=\frac{\sqrt{|p|^2  |s(\alpha)| ^2+ \frac14| \partial_\alpha (|s(\alpha)|^2)|^2}}{|\nabla H_1^{p}(|\alpha p + s(\alpha)|)|}  d\alpha d\theta .
\end{aligned}$$
Let us introduce the variable $ u = |w | = \sqrt{\alpha^2 |p|^2 + |s(\alpha)|^2}$. We compute 
$$ 2u du =\partial_\alpha |w_\alpha|^2 d\alpha$$ 
and hence \begin{equation}\label{bd-intSp2} 
\begin{aligned}
 \frac{ d \sigma(w)}{|\nabla H_1^{p}(w_\alpha)|} \; = C \frac{ \sqrt{|p|^2  |s(\alpha)| ^2+ \frac14| \partial_\alpha (|s(\alpha)|^2)|^2} }{2\partial_\alpha  |w_\alpha|^2|\nabla H_1^{p}(|\alpha p + s(\alpha)|)| } u du d\theta.
\end{aligned}\end{equation}
We recall that $H_1^{p}(w_\alpha) =0$ and hence 
$$0 =  \partial_\alpha w_\alpha \cdot \nabla_w H_1^{p} = \frac12 \partial_\alpha |w_\alpha|^2  \Big[ \frac{\mathcal{E}'(p+w_\alpha)}{|p+w_\alpha|} - \frac{\mathcal{E}'(w_\alpha) }{|w_\alpha|}\Big] + |p|^2 \frac{\mathcal{E}'(p+w_\alpha)}{|p+w_\alpha|} $$
which leads to \begin{equation}\label{bd-low123}|p|^2 \frac{\mathcal{E}'(p+w_\alpha)}{|p+w_\alpha|} = \frac12 \partial_\alpha|w_\alpha|^2 \left[\frac{\mathcal{E}'(w_\alpha) }{|w_\alpha|}-\frac{\mathcal{E}'(p+w_\alpha)}{|p+w_\alpha|}\right],\end{equation}
and 
\begin{equation}\label{bd-low123b}
\partial_\alpha |s(\alpha)|^2 = 2\frac{-\alpha|p|^2\frac{\mathcal{E}'(|w_\alpha|)}{|w_\alpha|}+(1+\alpha)|p|^2\frac{\mathcal{E}'(|p+w_\alpha|)}{|p+w_\alpha|}}{\left[\frac{\mathcal{E}'(w_\alpha) }{|w_\alpha|}-\frac{\mathcal{E}'(p+w_\alpha)}{|p+w_\alpha|}\right]}.
\end{equation}

We deduce from \eqref{bd-low123b} that
 \begin{equation*}
\begin{aligned}
 |\nabla  H_1^{p}|^2 \ =  &\  |p|^2\left[-\alpha   \frac{\mathcal{E}'(|w_\alpha|) }{|w_\alpha|} +(\alpha+1)  \frac{\mathcal{E}'(|p+w_\alpha|)}{|p+w_\alpha|}\right]^2 \ + |s(\alpha)|^2 \left[\frac{\mathcal{E}'(|p +w_\alpha|)}{|p+w_\alpha|} - \frac{\mathcal{E}'(|w_\alpha|) }{|w_\alpha|}\right]^2\\
  =  &\  \frac{\left|\partial_\gamma |s(\alpha)|^2\right|^2}{4|p|^2}\left[\frac{\mathcal{E}'(|p +w_\alpha|)}{|p+w_\alpha|} - \frac{\mathcal{E}'(|w_\alpha|) }{|w_\alpha|}\right]^2 \ + |s(\alpha)|^2 \left[\frac{\mathcal{E}'(|p +w_\alpha|)}{|p+w_\alpha|} - \frac{\mathcal{E}'(|w_\alpha|) }{|w_\alpha|}\right]^2,
\end{aligned}
\end{equation*}
which implies 
\begin{equation*}
\begin{aligned}
 \frac{d \sigma(w)}{|\nabla H_1^{p}(w)|} \;  = C \frac{|p|}{\partial_\alpha |w_\alpha|^2\left[\frac{\mathcal{E}'(|w_\alpha|)}{|w_\alpha|}-\frac{\mathcal{E}'(|p+w_\alpha|)}{|p+w_\alpha|}\right] } u dud\theta.
\end{aligned}\end{equation*}
The above and \eqref{bd-low123} yield
\begin{equation*}
\begin{aligned}
\frac{ d \sigma(w)}{|\nabla H_1^{p}(w)|} \;  = C \frac{u du}{2 |p| \frac{\mathcal{E}'(|p+w|)}{|p+w|}} .
\end{aligned}\end{equation*}
We the obtain
\begin{equation*}
\begin{aligned}
\int_{S_p^1} \frac{ F(|w|)}{|\nabla H_1^{p}(w)|} \; d \sigma(w) = C|p|^{d-2-\gamma}\int_0^\infty  F(u)(1+u/|p|)^{2-\gamma} du.
\end{aligned}\end{equation*}
The lemma follows by the spherical coordinates $dw=|w|^{d-1}d(|w|)d\sigma(\mathbb{S}^d).$ 
\end{proof}

\section{Weighted $L^1$ estimates}\label{Sec:MomentPropa}
In this section, we shall derive uniform estimates on the weighted $L^1$ norm, where weights are $n^{th}$-order monomials of $\cE_k$, which are defined by 
\begin{equation}\label{Def:MomentOrderk}
\mathfrak{M}_n[g]=\int_{\mathbb{R}^d}\cE^n_kg(k)dk
\end{equation}
in which we recall the energy function $\cE_k = |k|^{3/2}$. We stress that our estimates might depend on the positive coefficient of viscosity, but is independent of $\varrho$, in the equation; see \eqref{WeakTurbulenceGeneralized}.

\subsection{Estimate of the collision operator}
We first obtain the following estimate on the collision operator $Q[g]$. 

\begin{lemma}\label{Propo:C12} Let $N>1$. For any positive and radial function $g(p)=g(|p|)$, there exists a constant $C_N$, depending on $N$, such that the following holds 
\begin{equation}\label{Propo:C12:1}\begin{aligned}
& \int_{\mathbb{R}^d}Q[g](k)\cE^N_k \; dk\\
&\le     C_N   \sum_{n=[N/2]}^{N-1}  \Big( \mathfrak{M}_{n+\frac{7-2d}{3}}[g]   \mathfrak{M}_{N-n+\frac{2d-1}{3}}[g]  +  \mathfrak{M}_{n+\frac{11-2d}{3}}[g]   \mathfrak{M}_{N-n+\frac{2d-5}{3}}[g]\Big).
\end{aligned}
\end{equation}
\end{lemma}
\begin{proof} It is sufficient to prove the lemma for $N$ to be natural numbers. 
By Lemma \ref{Lemma:WeakFormulation}, we have 
$$
\begin{aligned}
\int_{\RR^d}Q[g](k) \cE_k^N \; dk  \ = &\  \iiint_{\mathbb{R}^{3d}} R_{k,k_1,k_2} [g]  \Big [\cE^N_{k}-\cE^N_{k_1}-\cE^N_{k_2} \Big] dk dk_1dk_2 . 
\end{aligned}
$$
Using the resonant conditions  $k = k_1 + k_2$ and $\cE_k = \cE_{k_1} + \cE_{k_2}$, dictated by the Dirac delta functions in $ R_{k,k_1,k_2} [g] $, we can write 
\begin{eqnarray*}
\cE_{k}^N-\cE^N_{k_1}-\cE^N_{k_2} = (\cE_{k_1} + \cE_{k_2})^N-\cE^N_{k_1}-\cE^N_{k_2}=\sum_{n=1}^{N-1} {{N}\choose{n}}\cE^n_{k_1}\cE^{N-n}_{k_2}. 
\end{eqnarray*}
Thus, we obtain 
\begin{equation}\label{QQQ}
\begin{aligned}
\int_{\RR^d}Q[g](k) \cE_k^N \; dk 
 &=  \iint_{\mathbb{R}^{2d}} R_{k_1+k_2,k_1,k_2} [g] \sum_{n=1}^{N-1} {{N}\choose{n}}\cE^n_{k_1}\cE^{N-n}_{k_2} dk_1dk_2 
\\
 &= 
  \iint_{\mathbb{R}^{2d}} R_{k_1+k_2,k_1,k_2} [g] \sum_{n=1}^{[N/2]-1} {{N}\choose{n}}\cE^n_{k_1}\cE^{N-n}_{k_2} dk_1dk_2 
 \\&\quad +  \iint_{\mathbb{R}^{2d}} R_{k_1+k_2,k_1,k_2} [g] \sum_{n=[N/2]}^{N-1} {{N}\choose{n}}\cE^n_{k_1}\cE^{N-n}_{k_2} dk_1dk_2 
\\
 &= :
I_1 + I_2.
 \end{aligned}
\end{equation}
Clearly, due to the symmetry of $k_1$ and $k_2$, it suffices to give estimates on $I_2$. Indeed, we write 
$$
\begin{aligned}
I_1 &=   \iint_{\mathbb{R}^{2d}} R_{k_1+k_2,k_1,k_2} [g] \sum_{n=N-[N/2]+1}^{N} {{N}\choose{N-n}}\cE^{N-n}_{k_1}\cE^{n}_{k_2} dk_1dk_2 ,
\end{aligned}$$
which is in fact $I_2$. We now estimate $I_2$. Recall that 
\begin{equation}\label{re-Kp}\begin{aligned}
R_{k,k_1,k_2} [g]=  4\pi |V_{k,k_1,k_2}|^2\delta(k-k_1-k_2)\delta(\mathcal{E}_k -\mathcal{E}_{k_1}-\mathcal{E}_{k_2})(g_1g_2-gg_1-gg_2)
\end{aligned}
\end{equation}
and the energy surface $S_k'$ is defined as in \eqref{def-Sp}. Thus, using the nonnegativity of $g(k)$, we can drop the last two terms $gg_1 + gg_2$ in \eqref{re-Kp}, yielding 
$$
\begin{aligned}
I_2 & \le  4\pi  \int_{\RR^d}\int_{S_{k_2}'} |V_{k_1+k_2,k_1,k_2}|^2 g_1 g_2 \sum_{n=[N/2]}^{N-1} {{N}\choose{n}}\cE^n_{k_1}\cE^{N-n}_{k_2} \frac{d\sigma(k_1)}{|\nabla H^{k_2}_1(k_1)|}dk_2 ,
\end{aligned}
$$
in which $H_1^k$ is defined as in \eqref{FunctionH}. Let us now estimate the collision kernel $V_{k_1+k_2,k_1,k_2}$, defined as in \eqref{def-VV}. We recall
$$
V_{k,k_1,k_2}  \ = \frac{1}{8\pi \sqrt{2\sigma}} \sqrt{\cE_k \cE_{k_1} \cE_{k_2}} \Big( \frac{L_{k_1,k_2}}{ |k| \sqrt{|k_1| |k_2|}} - \frac{L_{k,-k_1}}{ |k_2| \sqrt{|k| |k_1|}} - \frac{L_{k,-k_2}}{ |k_1| \sqrt{|k| |k_2|}} \Big) 
$$
with $L_{k_1,k_2} = k_1 \cdot k_2 +|k_1| |k_2|$ and $\cE_k = |k|^{3/2}$. It is clear that $|L_{k_1,k_2}|\le 2|k_1| |k_2|$. In addition, the energy identity $\cE_{k}=\cE_{k_1}+\cE_{k_2}$ in particular implies that $|k_1|\le |k|$ and $|k_2|\le |k|$, due to the monotonicity of $\cE_k$. Hence, for $k = k_1 + k_2$, we compute 
\begin{equation}\label{comp-L}
\begin{aligned}
 0 &\le L_{k,-k_1} = |k| |k_1| - k \cdot k_1 = |k_1| ( |k|- |k_1|) - k_1\cdot k_2
 \le 2 |k_1| |k_2|.
 \end{aligned}\end{equation}
The same bound holds for $L_{k,-k_2}$. This proves that 
$$
\begin{aligned}
|V_{k,k_1,k_2}| \le C_0  \sqrt{\cE_k \cE_{k_1} \cE_{k_2}}  \Big( \frac{\sqrt{|k_1| |k_2|}}{ |k| } +  \frac{\sqrt{|k_1|}}{ \sqrt{|k|} } + \frac{\sqrt{|k_2|}}{ \sqrt{|k|} } \Big) 
\end{aligned}$$
for some universal constant $C_0$. Using again $|k|\ge \max\{ |k_1|, |k_2|\}$, we obtain 
\begin{equation}\label{est-Vker}
|V_{k,k_1,k_2} |  \le C_0  \sqrt{\cE_{k} \cE_{k_1} \cE_{k_2}} 
\end{equation}
for all $(k,k_1,k_2)$ satisfying the resonant conditions  $k = k_1+k_2$ and $\cE_k = \cE_{k_1}+ \cE_{k_2}$. 
Hence, we have
$$
\begin{aligned}
I_2
  & \le  C_0  \int_{\RR^d}\int_{S_{k_2}'} \cE_{k_1}\cE_{k_2} (\cE_{k_1} + \cE_{k_2})g_1 g_2 \sum_{n=[N/2]}^{N-1} {{N}\choose{n}}\cE^n_{k_1}\cE^{N-n}_{k_2}\frac{d\sigma(k_1)}{|\nabla H^{k_2}_1(k_1)|}dk_2 
  \\
 & \le  C_0   \sum_{n=[N/2]}^{N-1} {{N}\choose{n}} \int_{\RR^d} g_2 \cE^{N-n+1}_{k_2}  \Big( \int_{S_{k_2}'}  \cE^{n+1}_{k_2} (\cE_{k_1} + \cE_{k_2})g_1\frac{d\sigma(k_1)}{|\nabla H^{k_2}_1(k_1)|}\Big) dk_2 .
  \end{aligned}$$

Next, applying Lemma \ref{lem-Sp1}, with $\gamma = 3/2$, to the surface integral on $S_{k_2}'$ and recalling that $g_1 = g(|k_1|)$, we obtain 
 \begin{equation}\label{SumMoment}\begin{aligned}
 &\int_{S_{k_2}'} g_1\cE^{n+1}_{k_1} (\cE_{k_1} + \cE_{k_2})\frac{d\sigma(k_1)}{|\nabla H^{k_2}_1(k_1)|}
\\ &\le C_0 \int_{\RR^d} g(k) \cE^{n+1}_{k_1} (\cE_{k_1} + \cE_{k_2})\cE_{k_2}^{\frac{2(d-4)}{3}}  (\cE_{k_1}^\frac13 + \cE_{k_2}^\frac13) \cE_{k_1}^{\frac{2(2-d)}{3}} dk_1.
  \end{aligned}\end{equation}

This proves 
$$
\begin{aligned}
I_2
 & \le    
 C_N   \sum_{n=[N/2]}^{N-1} \Big( \mathfrak{M}_{n+1+\frac{4-2d}{3}}[g]   \mathfrak{M}_{N-n+1+\frac{2d-4}{3}}[g]  +  \mathfrak{M}_{n+1+\frac{8-2d}{3}}[g]   \mathfrak{M}_{N-n+1+\frac{2d-8}{3}}[g]\Big).
   \end{aligned}
$$
This proves the lemma. 
%
\end{proof}

\begin{remark}\label{rem-radial} We note that by writing $L_{k,-k_1}$ and $L_{k,-k_2}$ as in \eqref{comp-L}, the kernel $|V_{k,k_1,k_2}|$ is radial in $k$. 
\end{remark}

\subsection{Weighted $L^1_N$ $(N>1)$ estimates}

\begin{proposition}\label{Propo:MomentsPropa} Let $N>1$. Suppose that $f_0(k)=f_0(|k|)$ is a nonnegative radial initial data
satisfying 
$$\int_{\mathbb{R}^d}f_0(k) (\mathcal{E}_k + \cE_k^N) \;dk<\infty.$$
Then,  corresponding nonnegative radial solutions $f(t,k)=f(t,|k|)$ of \eqref{WeakTurbulenceGeneralized}, with $f(0,k) = f_0(k)$, satisfy
\begin{equation}\label{EE-bound}\sup_{t\ge 0} \int_{\mathbb{R}^d}f(t,k)\cE_k^N\;dk \le C_{N}
\end{equation}
for some finite constant $C_{N}$ depending on the initial data and the viscosity. 
 \end{proposition}

We need the following simple lemma. 
\begin{lemma}\label{lem-Holder} For $M>n>p$, there holds
\begin{equation}\label{Holder} \mathfrak{M}_n[g] \le \mathfrak{M}_p^{\frac{M-n}{M-p}}[g] \mathfrak{M}_{M}^{\frac{n-p}{M-p}}[g].\end{equation}
\end{lemma}
\begin{proof}
The lemma follows from the definition of $\mathfrak{M}_n$ and the following H\"older inequality
\begin{eqnarray*}
\int_{\mathbb{R}^d} g(k)\cE_k^n dk &\le& \left( \int_{\mathbb{R}^d} g(k)\cE_k^p dk \right)^{\frac{M-n}{M-p}} \left( \int_{\mathbb{R}^d} g(k)\cE_k^M dk\right)^{\frac{n-p}{M-p}} .
\end{eqnarray*}
\end{proof}
\begin{proof}[Proof of Proposition \ref{Propo:MomentsPropa}]
Using $\varphi = \cE_{k}^N$ as a test function in \eqref{WeakTurbulenceGeneralized}, we obtain 
$$
\frac{d}{dt}\int_{\mathbb{R}^d}f(t,k)\cE^N_{k}dk + 2\nu  \int_{\RR^d} (|k|^2+\varrho|k|^4) f(t,k)\cE_k^N \; dk =\int_{\mathbb{R}^d}Q[f](k)\cE_{k}^Ndk.
$$
By using Proposition \ref{Propo:C12} and recalling the definition of $\mathfrak{M}_M$, the above yields 
$$\begin{aligned}
 & \frac{d}{dt} \mathfrak{M}_{N}[f(t)] + 2\nu   \mathfrak{M}_{N+\frac43}[f(t)] \le 
\\&\le  
 C_N   \sum_{n=[N/2]}^{N-1}  \Big( \mathfrak{M}_{n+\frac{7-2d}{3}}[f(t)]   \mathfrak{M}_{N-n+\frac{2d-1}{3}}[f(t)]  +  \mathfrak{M}_{n+\frac{11-2d}{3}}[f(t)]   \mathfrak{M}_{N-n+\frac{2d-5}{3}}[f(t)]\Big).
\end{aligned}$$
Now using Lemma \ref{lem-Holder}, with $p=7/3$ and $M = N $, we get 
$$
\begin{aligned}
 \sum_{n=[N/2]}^{N-1} \mathfrak{M}_{n+\frac{7-2d}{3}}[f(t)]   \mathfrak{M}_{N-n+\frac{2d-1}{3}}[f(t)] 
  & \le  
 \sum_{n=[N/2]}^{N-1}\mathfrak{M}_{2} [f(t)]\mathfrak{M}_N[f(t)]
\\  & \le  C_N \mathfrak{M}_{2} [f(t)]\mathfrak{M}_N[f(t)]
    \end{aligned}
$$
and 
$$
\begin{aligned}
 \sum_{n=[N/2]}^{N-2} \mathfrak{M}_{n+\frac{7-2d}{3}}[f(t)]   \mathfrak{M}_{N-n+\frac{2d-1}{3}}[f(t)]
  & \le  
 \sum_{n=[N/2]}^{N-2}\mathfrak{M}_{2} [f(t)]\mathfrak{M}_N[f(t)]
\\  & \le  C_N \mathfrak{M}_{2} [f(t)]\mathfrak{M}_N[f(t)].
    \end{aligned}
$$
We obtain \begin{equation}\label{ineq-EE}\frac{d}{dt}  \mathfrak{M}_{N}[f(t)] + \nu   \mathfrak{M}_{N+\frac43}[f(t)] \le  C_N   \mathfrak{M}_{2} [f(t)]\mathfrak{M}_{N}[f(t)].
\end{equation}
The uniform boundedness of $\mathfrak{M}_{N}[f(t)] $ follows from the standard Gronwall's lemma, upon using the following energy inequality (see \eqref{Coro:ConservatioEnergy}):
$$ \mathfrak{M}_1[f(t)]   + \int_0^t \mathfrak{M}_{2}[f(s)]\; ds \le \mathfrak{M}_1[f_0] .$$
The proposition is proved. 
\end{proof}

\subsection{Weighted $L^1_\frac13$  estimates}

\begin{proposition}\label{Lemma:13Norm} Let $f_0(k)=f_0(|k|)$ be nonnegative and satisfy  
$$\int_{\mathbb{R}^d}f_0(k) \cE_k^{1/3} (1 + \cE_k^{5/3}) \;dk<\infty.$$
Then, corresponding nonnegative radial solutions $f(t,k)=f(t,|k|)$ of \eqref{WeakTurbulenceGeneralized}, with $f(0,k) = f_0(k)$, satisfy
\begin{equation}\label{EE-bound} \int_{\mathbb{R}^d}f(t,k)\cE_k^\frac13\;dk \le c_0e^{c_1t}, \qquad \forall t\ge 0,
\end{equation}
for some universal constants $c_0$, $c_1$  depending on the initial data and the viscosity.
\end{proposition}
\begin{proof} By Proposition \ref{Propo:MomentsPropa}, the $L_s^1$-norm of $f$ is bounded for $s\in[1,\frac73]$.
Using $\mathcal{E}_k^{1/3}$ as a test function in \eqref{WeakTurbulenceGeneralized}, we obtain

\begin{align}\label{Lemma:13Norm:E1}
\frac{d}{dt} & \mathfrak{M}_{\frac{1}{3}}[f(t)] + 2\nu   \mathfrak{M}_{\frac53}[f(t)] 
\le \int_{\mathbb{R}^d}Q[f](k)\cE_{k}^\frac13dk.
\end{align}
We now divide the proof into two steps.

~\\
{\bf Step 1: Estimating the collision integral.} We can estimate  the right hand side of \eqref{Lemma:13Norm:E1} as
$$ 
\begin{aligned}
 \int_{\mathbb{R}^d} \cE_k^\frac13Q[f](k)dk
& \le \iiint_{\mathbb{R}^{3d}}  |R_{k,k_1,k_2}[f]| \Big( \cE_k^\frac13 + \cE_{k_1}^\frac13 + \cE_{k_2}^\frac13\Big) \; dkdk_1dk_2,
\end{aligned}
$$
in which,  recall  $$\begin{aligned}
R_{k,k_1,k_2} [g]=  4\pi|V_{k,k_1,k_2}|^2\delta(k-k_1-k_2)\delta(\mathcal{E}_k -\mathcal{E}_{k_1}-\mathcal{E}_{k_2})(f_1f_2-ff_1-ff_2) .
\end{aligned}
$$
By the resonant conditions  $k = k_1 + k_2$ and $\cE_k = \cE_{k_1} + \cE_{k_2}$, we  the  integrals can be re-expressed in terms of the surface integrals over $\RR^d\times S_k$ and $\RR^d \times S_{k_1}'$, as follows
$$ 
\begin{aligned}
\int_{\mathbb{R}^d} \cE_k^\frac13Q[f](k)dk  
&\le 
C\int_{\mathbb{R}^d}\int_{S_{k_1}'}  |V_{k_1+k_2,k_1,k_2}|^2 |f_1f_2|
\Big( \cE_{k_1+k_2}^\frac13 + \cE_{k_1}^\frac13 + \cE_{k_2}^\frac13\Big)  \;   \frac{d \sigma(k_2) }{|\nabla H_1^{k_1}(k_2)|} dk_1
\\
&\quad + 
C\int_{\mathbb{R}^d}\int_{S_k}  |V_{k,k - k_2,k_2}|^2 | ff_2| \Big( \cE_k^\frac13 + \cE_{k-k_2}^\frac13 + \cE_{k_2}^\frac13\Big)  \;   \frac{d \sigma(k_2) }{|\nabla H_0^{k}(k_2)|}dk
\\&=: I_1 + I_2.\end{aligned}
$$

~\\
{\em Estimate on $I_1$.} By the Cauchy-Schwarz inequality and the conservation law $\cE_{k_1+k_2}= \cE_{k_1} + \cE_{k_2}$,  
$$ 
\begin{aligned}
\cE_{k_1+k_2}^\frac13 \le C_N (\cE_{k_1}^\frac13+ \cE_{k_2}^\frac13) \end{aligned}
$$
which then leads to
\begin{equation*}
\begin{aligned}
I_1 
&\le  
C \int_{\mathbb{R}^d}\int_{S_{k_2}'}  |V_{k_1+k_2,k_1,k_2}|^2 |f_1| |f_2|
 (\cE_{k_1}^\frac13 + \cE_{k_2}^\frac13)  \;   \frac{d \sigma(k_1) }{|\nabla H_1^{k_2}(k_1)|}dk_2.
\end{aligned}\end{equation*}
Let us note from \eqref{est-Vker} that $|V_{k,k_1,k_2} |^2  \le C_0\cE_k \cE_{k_1}\cE_{k_2}.$ By Lemma \ref{lem-Sp1} and the same argument used for \eqref{SumMoment}, $I_1$ can be bounded the following way 
\begin{equation}\label{Lemma:13Norm:E2}
\begin{aligned}
I_1 \ \le \  &\int_{\mathbb{R}^d}\int_{S_{k_2}'}  |V_{k_1+k_2,k_1,k_2}|^2 | f_1| |f_2|
 (\cE_{k_1}^\frac13 + \cE_{k_2}^\frac13)  \;   \frac{d \sigma(k_1) }{|\nabla H_1^{k_2}(k_1)|} dk_2
\\&\le C \int_{\mathbb{R}^d}\int_{S_{k_2}'}  \cE_{k_1} \cE_{k_2} | f_1| |f_2|
 (\cE_{k_1}^{\frac13} + \cE_{k_2}^{\frac13}) (\cE_{k_1} + \cE_{k_2})  \;   \frac{d \sigma(k_1) }{|\nabla H_1^{k_2}(k_1)|}dk_2
\\
&\le C\iint_{\mathbb{R}^{2d}} \cE_{k_1}\cE_{k_2}  | f_1| |f_2|
 (\cE_{k_1}^{\frac13} + \cE_{k_2}^{\frac13})  (\cE_{k_1}^{\frac{2d-4}{3}}\cE_{k_2}^{\frac{4-2d}{3}} + \cE_{k_1}^{\frac{2d-8}{3}}\cE_{k_2}^{\frac{8-2d}{3}})  \;  dk_1dk_2
\\
&\le C 7\|f\|_{L^1_\frac13}\|f\|_{L^1_2}.
\end{aligned}\end{equation}

~\\
{\em Estimate on $I_2$.} We turn to estimate $I_2$. Again, recalling $\cE_k = \cE_{k-k_2} + \cE_{k_2} \ge \max\{ \cE_{k-k_2}, \cE_{k_2}\}$, we bound 
\begin{equation*}
\begin{aligned}
I_2 
&\le C \int_{\mathbb{R}^d}\int_{S_k}  |V_{k,k - k_2,k_2}|^2 |ff_2| \Big( \cE_k^\frac13 + \cE_{k-k_2}^\frac13 + \cE_{k_2}^\frac13\Big)  \;  \frac{d \sigma(k_2) }{|\nabla H_0^{k}(k_2)|}dk.
\end{aligned}\end{equation*}
Notice that $|V_{k,k - k_2,k_2}|^2 \le C_0 \cE_k \cE_{k-k_2} \cE_{k_2} \le C_0 \cE_k^2 \cE_{k_2}$. By Lemma \ref{lem-Sp},  the following holds true
\begin{equation}\label{Lemma:13Norm:E3}
\begin{aligned}
&
\int_{\mathbb{R}^d}\int_{S_k}  |V_{k,k - k_2,k_2}|^2 |f| |f_2|  \cE_k^\frac13   \;  \frac{d \sigma(k_2) }{|\nabla H_0^{k}(k_2)|}dk
\\
&\le C_0 \int_{\mathbb{R}^d}\int_{S_k}  \cE_k^2 \cE_{k_2} |f| |f_2| \cE_k^\frac13 \;  \frac{d \sigma(k_2) }{|\nabla H_0^{k}(k_2)|}dk
\\
&\le C_0 \iint_{\mathbb{R}^{2d}} \cE_{k_2} |k_2|^{2-d}|f| |f_2|  \cE_k^{\frac73}|k|^{d-\frac72}  \;  dk dk_2
\\
&\le C_M\|f\|_{L^1_{\frac{2d}{3}}}\|f\|_{L^1_{\frac{7-2d}{3}}}\le C_M\|f\|_{L^1_{2}}\|f\|_{L^1_{\frac13}}.
\end{aligned}\end{equation}

Combining \eqref{Lemma:13Norm:E2}-\eqref{Lemma:13Norm:E3} and using the fact that the $L_s^1$-norm of $f$ is bounded for $s\in[1,2]$, we obtain 
\begin{equation}\label{Lemma:13Norm:E4}
\begin{aligned}
\int_{\mathbb{R}^d} \cE_k^\frac13Q[f](k)dk  
&\le 
C^*\|f\|_{L^1_{\frac13}}.\end{aligned}
\end{equation}

~\\
{\bf Step 2: Estimating the $L^1_{\frac13}$-norm.} Putting together the two estimates \eqref{Lemma:13Norm:E1} and \eqref{Lemma:13Norm:E4} yields

\begin{align}\label{Lemma:13Norm:E5}
\frac{d}{dt} & \mathfrak{M}_{\frac{1}{3}}[f(t)] + 2\nu   \mathfrak{M}_{\frac53}[f(t)] 
\le C^*\mathfrak{M}_{\frac{1}{3}}[f(t)],
\end{align}
which implies the bound on $\mathfrak{M}_{\frac{1}{3}}[f(t)]$.

The proof of the lemma is complete. 
\end{proof}

\section{$L^2$ estimates}\label{Sec:L2}

\begin{proposition}\label{Pro:L2}
Suppose that $f_0(k)=f_0(|k|)$ is a nonnegative radial initial data
with 
$$\int_{\mathbb{R}^d}f_0(k) \mathcal{E}_k (1+\cE_k^{{\frac{2d-4}{3}}})\;dk<\infty$$
and 
$$\int_{\mathbb{R}^d}|f_0(k)|^2\;dk<\infty.$$
Then, corresponding nonnegative radial solutions $f(t,k)=f(t,|k|)$ of \eqref{WeakTurbulenceGeneralized}, with $f(0,k) = f_0(k)$, satisfy
\begin{equation}\label{EE-bound}\int_{\mathbb{R}^d}|f(t,k)|^2\;dk \le c_0e^{c_1t}.
\end{equation}
for some universal constants $c_0$, $c_1$  depending on the initial data and the viscosity.
\end{proposition}
\begin{proof}
Using $f$ as a test function in \eqref{WeakTurbulenceGeneralized}, we obtain the following identity
\begin{equation}\label{Pro:L2:E1}
\frac{1}{2}\frac{d}{dt}\int_{\mathbb{R}^{d}}f^2dk  \ + \ \nu\int_{\mathbb{R}^{d}}(|k|^2+\varrho |k|^4)f^2dk \  = \ \int_{\mathbb{R}^{d}}Q[f]fdk. 
\end{equation}
As an application of Lemma \ref{Lemma:WeakFormulation}, the right hand side of \eqref{Pro:L2:E1} could be expressed as
\begin{equation}\label{Pro:L2:E2}
\begin{aligned}
\int_{\mathbb{R}^{d}}Q[f]fdk \  &= \ 4\pi\iiint_{\mathbb{R}^{3d}}|V_{k,k_1,k_2}|^2\delta(k-k_1-k_2)\delta(\mathcal{E}_k-\mathcal{E}_{k_1}-\mathcal{E}_{k_2})
\\&\quad \times (f_1f_2-ff_1-ff_2)(f-f_1-f_2)dk_1dk_2dk. 
\end{aligned}\end{equation}
By taking into account the positivity of $f$, 
the term inside the integral of \eqref{Pro:L2:E2} can be bounded by removing all the terms containing the negative sign, giving 
\begin{equation*}
\begin{aligned}
(f_1f_2-ff_1-ff_2)(f-f_1-f_2) \ 
&\le 3ff_1f_2 \ + \ ff_1^2 \ + \ ff_2^2
\\&\le \frac 52 f (f_1^2 \ + \ f_2^2).  
\end{aligned}\end{equation*}
Inserting the  above inequality into \eqref{Pro:L2:E2} and using the symmetry in $k_1$ and $k_2$, we find
\begin{equation*}
\begin{aligned}
\int_{\mathbb{R}^{d}}Q[f]fdk \ &\le  \ C\iiint_{\mathbb{R}^{3d}}|V_{k,k_1,k_2}|^2\delta(k-k_1-k_2)\delta(\mathcal{E}_k-\mathcal{E}_{k_1}-\mathcal{E}_{k_2}) f(f_1^2 \ + \ f_2^2)dk_1dk_2dk
\\
 \ &\le  \ C\iiint_{\mathbb{R}^{3d}}|V_{k,k_1,k_2}|^2\delta(k-k_1-k_2)\delta(\mathcal{E}_k-\mathcal{E}_{k_1}-\mathcal{E}_{k_2})ff_2^2dk_1dk_2dk. 
\end{aligned}\end{equation*}
Then, again using the definition of the Dirac functions $ \delta(k-k_1-k_2)$ and $\delta(\mathcal{E}_k-\mathcal{E}_{k_1}-\mathcal{E}_{k_2})$, we obtain 
\begin{equation*}
\int_{\mathbb{R}^{d}}Q[f]fdk \ \le  \ C\int_{\mathbb{R}^{d}}\int_{S_k}|V_{k,k_1,k_2}|^2ff_2^2 \; \frac{d \sigma(k_2) }{|\nabla H_0^{k}(k_2)|}dk. 
\end{equation*}
Recall that $|V_{k,k_1,k_2}|^2$ is bounded by $C\mathcal{E}_k\mathcal{E}_{k_1}\mathcal{E}_{k_2}$, and on the surface $S_k$, $\mathcal{E}_{k-k_2}\le \mathcal{E}_k$ and $\mathcal{E}_{k_2}\le \mathcal{E}_k$. This together with Lemma \ref{lem-Sp} yields
\begin{equation*}
\begin{aligned}
\int_{\mathbb{R}^{d}}Q[f]fdk \ &\le  \ C\int_{\mathbb{R}^{d}}\int_{S_k}\mathcal{E}_k^2\mathcal{E}_{k_2}ff_2^2\; \frac{d \sigma(k_2) }{|\nabla H_0^{k}(k_2)|} dk. 
\\
\ &\le  \ C\iint_{\mathbb{R}^{2d}}\mathcal{E}_k^2\mathcal{E}_{k_2}|k|^{-\frac72+d}|k_2|^{2-d}ff_2^2dk_2dk
\\
\ &\le  \ C\Big(\int_{\mathbb{R}^{d}} \mathcal{E}_k^\frac{2d-1}{3}f\; dk \Big) \Big( \int_{\RR^d} |k|^{\frac{7-2d}{3}} |f|^2 \;dk\Big). 
\end{aligned}
\end{equation*}
Now, by  interpolating the results of Proposition \ref{Propo:MomentsPropa}, the $L_{\frac{2d-1}{3}}^1$ norm of $f$ is bounded. Hence, 
\begin{equation}\label{Pro:L2:E3}
\int_{\mathbb{R}^{d}}Q[f]fdk \ \le  \ C\int_{\mathbb{R}^{d}}|k|^\frac12f^2dk. 
\end{equation}
Putting this into \eqref{Pro:L2:E1} yields
\begin{equation}\label{Pro:L2:E4}
\frac{d}{dt}\int_{\mathbb{R}^{d}}f^2dk \  \le \ \int_{\mathbb{R}^{d}}\left(C|k|^{\frac{7-2d}{3}}- \nu |k|^2\right)f^2dk. 
\end{equation}
Let us note that the function $\rho(x)=C x^{\frac{7-2d}{3}}-\nu x^2$, $d=2,3$, $x\in\mathbb{R}_+$ is bounded from above by some positive constant $C_1$ (depending on $\nu$). This 
proves \begin{equation}\label{Pro:L2:E5}
\frac{d}{dt}\int_{\mathbb{R}^{d}}f^2dk \  \le \ C_1\int_{\mathbb{R}^{d}}f^2dk, 
\end{equation}
which yields the proposition.
\end{proof}

\section{Holder estimates for $Q[f]$}\label{Sec:HolderEstimate}

In this section, we study the H\"older continuity of the collision operator $Q[f]$ with respect to weighted $L^1_N$ norm:
$$ \| f\|_{L^1_N} = \int_{\RR^d} f(k) \cE_k^N \; dk. $$

\begin{proposition}\label{Propo:HolderC12} Let $M,N\ge 1$, and let $\mathcal{S}_M$ be any bounded subset of $L_1^1(\RR^d) \cap L^1_{N+3}(\RR^d)$, with $L^1_1$ and $L^1_{N+3}$ norms bounded by $M$. Then, there exists a constant $C_{M,N}$, depending on $M,N$, so that 
\begin{equation}\label{Propo:HolderC12:1} 
\|Q[g]-Q[h]\|_{L^1_{N}} + \|Q[g]-Q[h]\|_{L^1_\frac13} \le C_{M,N}  \Big( \| g-h\|_{L^1_N} + \| g - h\|_{L^1_{\frac 13}} \Big)^{\frac13} 
\end{equation}
for all $g,h\in \mathcal{S}_M$. 
\end{proposition}

We first prove the following lemma. 
\begin{lemma}\label{lem-QQ} Let $M,N>0$, and let $\mathcal{S}_M$ be any bounded subset of $L_{\frac13}^1(\RR^d) \cap L^1_{N+2}(\RR^d)$, with $L^1_{\frac13}$ and $L^1_{N+2}$ norms bounded by $M$. Then, there exists a constant $C_{M,N}$, depending on $M,N$, so that 
\begin{equation}\label{Q-L1bound} 
\|Q[g]-Q[h]\|_{L^1_{N}} \le C_{M,N}  \Big( \| g-h\|_{L^1_{\frac13}} + \| g - h\|_{L^1_{N+2}} 
\Big)
\end{equation}
for all $g,h\in \mathcal{S}_M$. 
\end{lemma}
\begin{proof} 
By definition of the collision operator, we compute $$Q[g] - Q[h] = 
\iint_{\mathbb{R}^{2d}} \Big[ R_{k,k_1,k_2}[g] - R_{k,k_1,k_2}[h]   - 2  ( R_{k_1,k,k_2}[g] - R_{k_1,k,k_2}[h])  \Big] dk_1dk_2 $$
and hence 
$$ 
\begin{aligned}
\|Q[g]-Q[h]\|_{L^1_N}  & = \int_{\mathbb{R}^d} \cE_k^N|Q[g](k)-Q[h](k)|dk
\\
&\le \iiint_{\mathbb{R}^{3d}} \cE_k^N |R_{k,k_1,k_2}[g] - R_{k,k_1,k_2}[h]| \; dkdk_1dk_2
\\&\quad 
+  2 \iiint_{\mathbb{R}^{3d}} \cE_k^N |  R_{k_1,k,k_2}[g] - R_{k_1,k,k_2}[h] | dk dk_1dk_2 
\\& = \iiint_{\mathbb{R}^{3d}}  |R_{k,k_1,k_2}[g] - R_{k,k_1,k_2}[h]| \Big( \cE_k^N + \cE_{k_1}^N + \cE_{k_2}^N\Big) \; dkdk_1dk_2.
\end{aligned}
$$
Recall that $$\begin{aligned}
R_{k,k_1,k_2} [g]=  C|V_{k,k_1,k_2}|^2\delta(k-k_1-k_2)\delta(\mathcal{E}_k -\mathcal{E}_{k_1}-\mathcal{E}_{k_2})(g_1g_2-gg_1-gg_2) .
\end{aligned}
$$
Using the resonant conditions  $k = k_1 + k_2$ and $\cE_k = \cE_{k_1} + \cE_{k_2}$, we write the triple integrals in term of the surface integrals over $\RR^d\times S_k$ and $\RR^d \times S_{k_1}'$. It follows at once that 
$$ 
\begin{aligned}
\|Q[g]-Q[h]\|_{L^1_N}  
&\le 
C\int_{\mathbb{R}^d}\int_{S_{k_1}'}  |V_{k_1+k_2,k_1,k_2}|^2 | g_1g_2 - h_1h_2|
\Big( \cE_{k_1+k_2}^N + \cE_{k_1}^N + \cE_{k_2}^N\Big)  \;  \frac{d \sigma(k_2) }{|\nabla H_1^{k_1}(k_2)|} dk_1
\\
&\quad + 
8\pi \int_{\mathbb{R}^d}\int_{S_k}  |V_{k,k - k_2,k_2}|^2 | gg_2 - hh_2| \Big( \cE_k^N + \cE_{k-k_2}^N + \cE_{k_2}^N\Big)  \;   \frac{d \sigma(k_2) }{|\nabla H_0^{k}(k_2)|}  dk
\\&=: J_1 + J_2,\end{aligned}
$$
in which $H_j^k$ are defined as in \eqref{FunctionH}. 

~\\
{\bf Estimate on $J_1$.} Using the triangle inequality and the conservation law $\cE_{k_1+k_2}= \cE_{k_1} + \cE_{k_2}$, we have 
$$ 
\begin{aligned}
\cE_{k_1+k_2}^N \le C_N (\cE_{k_1}^N + \cE_{k_2}^N) \end{aligned}
$$
and 
$$ | g_1g_2 - h_1h_2| \le | g_1 - h_1 | |g_2| + |h_1| |g_2 - h_2|.$$
Thus, we obtain 
\begin{equation}\label{est-J11}
\begin{aligned}
J_1 
&\le  
C_N \int_{\mathbb{R}^d}\int_{S_{k_2}'}  |V_{k_1+k_2,k_1,k_2}|^2 | g_1 - h_1| |g_2|
 (\cE_{k_1}^N + \cE_{k_2}^N)  \;  \frac{d \sigma(k_1) }{|\nabla H_1^{k_2}(k_1)|}  dk_2
\\&\quad + 
 C_N \int_{\mathbb{R}^d}\int_{S_{k_1}'}  |V_{k_1+k_2,k_1,k_2}|^2 | h_1| | g_2 - h_2|
 (\cE_{k_1}^N + \cE_{k_2}^N)  \;  \frac{d \sigma(k_2) }{|\nabla H_1^{k_1}(k_2)|}  dk_1 .
\end{aligned}\end{equation}
Recall from \eqref{est-Vker} that $|V_{k,k_1,k_2} |^2  \le C_0\cE_k \cE_{k_1}\cE_{k_2}.$ Thus, together with Lemma \ref{lem-Sp1} and the same argument used for \eqref{SumMoment}, we estimate the first integral term in $J_1$, yielding  
$$
\begin{aligned}
 &\int_{\mathbb{R}^d}\int_{S_{k_2}'}  |V_{k_1+k_2,k_1,k_2}|^2 | g_1 - h_1| |g_2|
 (\cE_{k_1}^N + \cE_{k_2}^N)  \;   \frac{d \sigma(k_1) }{|\nabla H_1^{k_2}(k_1)|}  dk_2
\\&\le C_0 \int_{\mathbb{R}^d}\int_{S_{k_2}'}  \cE_{k_1} \cE_{k_2} | g_1 - h_1| |g_2|
 (\cE_{k_1}^{N} + \cE_{k_2}^{N})   (\cE_{k_1} + \cE_{k_2})\;    \frac{d \sigma(k_1) }{|\nabla H_1^{k_2}(k_1)|}  dk_2
\\
&\le C_0 \iint_{\mathbb{R}^{2d}} \cE_{k_1}\cE_{k_2}  (\cE_{k_1}^{\frac{2d-4}{3}}\cE_{k_2}^{\frac{4-2d}{3}} + \cE_{k_1}^{\frac{2d-8}{3}}\cE_{k_2}^{\frac{8-2d}{3}})| g_1 - h_1| |g_2|
 (\cE_{k_1}^{N} + \cE_{k_2}^{N})    \;  dk_1dk_2
\\
&\le C_M \Big( \| g-h\|_{L^1_{N+\frac{2d-1}{3}}} + \| g - h\|_{L^1_{\frac{7-2d}{3}}} \Big),
\end{aligned}$$
in which we have used the boundedness of $g$ in $L^1_{\frac{1}{3}} \cap L^1_{N+2}$. By symmetry, the same estimate holds for the second integral in $J_1$. 

~\\
{\bf Estimate on $J_2$.} We turn to estimate $J_2$. Again, using 
$$| gg_2 - hh_2| \le |g - h| |g_2| + |h| |g_2 - h_2| ,$$
and recalling $\cE_k = \cE_{k-k_2} + \cE_{k_2} \ge \max\{ \cE_{k-k_2}, \cE_{k_2}\}$, we estimate 
\begin{equation}\label{est-J22}
\begin{aligned}
J_2 
&= C \int_{\mathbb{R}^d}\int_{S_k}  |V_{k,k - k_2,k_2}|^2 | gg_2 - hh_2| \Big( \cE_k^N + \cE_{k-k_2}^N + \cE_{k_2}^N\Big)  \;  \frac{d \sigma(k_2) }{|\nabla H_0^{k}(k_2)|}  dk
\\&\le 
C_N  \int_{\mathbb{R}^d}\int_{S_k}  |V_{k,k - k_2,k_2}|^2 |g - h| |g_2| \cE_k^N \;  \frac{d \sigma(k_2) }{|\nabla H_0^{k}(k_2)|} dk
\\&\quad +C_N  \int_{\mathbb{R}^d}\int_{S_k}  |V_{k,k - k_2,k_2}|^2 |h| |g_2 - h_2| \cE_k^N \;  \frac{d \sigma(k_2) }{|\nabla H_0^{k}(k_2)|}  dk.
\end{aligned}\end{equation}
Recall that $|V_{k,k - k_2,k_2}|^2 \le C_0 \cE_k \cE_{k-k_2} \cE_{k_2} \le C_0 \cE_k^2 \cE_{k_2}$. Therefore, using Lemma \ref{lem-Sp} with $\gamma = 3/2$, 
we estimate   
$$\begin{aligned}
&
\int_{\mathbb{R}^d}\int_{S_k}  |V_{k,k - k_2,k_2}|^2 |g - h| |g_2|  \cE_k^N   \;  \frac{d \sigma(k_2) }{|\nabla H_0^{k}(k_2)|}  dk
\\
&\le C_0 \int_{\mathbb{R}^d}\int_{S_k}  \cE_k^2 \cE_{k_2} |g - h| |g_2| \cE_k^N \;  \frac{d \sigma(k_2) }{|\nabla H_0^{k}(k_2)|}  dk
\\
&\le C_0 \iint_{\mathbb{R}^{2d}} \cE_{k_2} |k_2|^{2-d}|g - h| |g_2||k|^{d-\frac72}  \cE_k^{N+2}  \;  dk dk_2
\\
&\le C_M\| g-h\|_{L^1_{N+\frac{2d-1}{3}}} ,
\end{aligned}$$
in which we have again used the boundedness of $g$ with respect to $L^1_{\frac{7-2d}{3}}$ norm.

We now estimate the second integral in $J_2$. 
$$\begin{aligned}
& 
 \int_{\mathbb{R}^d}\int_{S_k}  |V_{k,k - k_2,k_2}|^2 |h| |g_2 - h_2| \cE_k^N \;  \frac{d \sigma(k_2) }{|\nabla H_0^{k}(k_2)|}  dk
\\
&\le C_0 \int_{\mathbb{R}^d}\int_{S_k} \cE_k^2 \cE_{k_2}|h| |g_2 - h_2| \cE_k^N \;  \frac{d \sigma(k_2) }{|\nabla H_0^{k}(k_2)|}  dk
\\&\le C_N \| g - h\|_{L^1_{{\frac{7-2d}{3}}} }
\end{aligned}$$
in which again the boundedness of $h$ in $L^1_{N+\frac{2d-1}{3}}$ was used. 

Combining, we obtain 
$$
\|Q[g]-Q[h]\|_{L^1_N}  \le  C_{M,N} \Big( \| g-h\|_{L^1_{N+\frac{2d-1}{3}}} + \| g - h\|_{L^1_{\frac{7-2d}{3}} } \Big) .$$
Since $\cE_k^{N+\frac{2d-1}{3}} + \cE_k^{\frac{7-2d}{3}} \le C( \cE_k^\frac13 + \cE_k^{N+2})$, the above reduces to 
\begin{equation}\label{bound-Qgh}\|Q[g]-Q[h]\|_{L^1_N}  \le  C_{M,N} \Big( \| g - h\|_{L^1_{\frac{1}{3}} } + \| g-h\|_{L^1_{N+2}} \Big) .\end{equation} 
The proof of the lemma is complete. 
\end{proof}

~\\
\begin{proof}[Proof of Proposition \ref{Propo:HolderC12}] The proposition now follows straightforwardly from the previous lemma. Indeed, we recall the interpolation inequality (see Lemma \ref{lem-Holder}):
$$ \| g\|_{L^1_n} \le \| g \|_{L^1_p}^{\frac{q-n}{q-p}}  \| g \|_{L^1_q}^{\frac{n-p}{q-p}}  $$
for $q>n>p$. Together with the boundedness of $g,h$ in $L^1_1 \cap L^1_{N+3}$, we obtain 
$$
\begin{aligned}
 \| g-h\|_{L^1_{N+2}}  &\le \| g-h\|_{L^1_N}^{\frac13} \| g-h\|_{L^1_{N+3}}^{\frac23} \le C_M \| g-h\|_{L^1_N}^{\frac13} 
\end{aligned}$$
Lemma \ref{lem-QQ} yields 
$$\|Q[g]-Q[h]\|_{L^1_N}  \le  C_{M,N} \Big( \| g-h\|_{L^1_N} + \| g - h\|_{L^1_{\frac13}} \Big)^{\frac13} $$
which holds for all $N >0$. In particular, the above holds for $\|Q[g]-Q[h]\|_{L^1_\frac13} $. The proposition follows. 
\end{proof}

\section{Proof of Theorem \ref{Theorem:Main}}\label{Sec:Main}
\subsection{Case 1: $\varrho>0$}
The proof of our main theorem, Theorem \ref{Theorem:Main}, for the case $\varrho>0$ uses the following abstract theorem, introduced in \cite{AlonsoGambaBinh,SofferBinh1}  inspired by the previous works of \cite{Bressan,Martin}. For sake of completeness, the proof of the abstract theorem will be given in the Appendix.

\begin{theorem}\label{Theorem:ODE} Let $[0,T]$ be a time interval, $E:=(E,\|\cdot\|)$ be a Banach space, $\mathcal{S}$  be a bounded, convex and closed subset of $E$, and $Q:\mathcal{S}\rightarrow E$ be an operator  satisfying the following properties:
\begin{itemize}
\item [$(\mathfrak{A})$] Let $\|\cdot\|_*$ be a different norm of $E$, satisfying $\|\cdot\|_*\leq C_E\|\cdot\|$ for some universal constant $C_E$, and the function
\begin{eqnarray*}
|\cdot |_*: E&\longrightarrow&\mathbb{R}\\
  u&\longrightarrow& |u|_*,
\end{eqnarray*}
satisfying $$|u+v|_*\le |u|_*+|v|_*, \mbox{ and } \ \ \ |\alpha u|_*=\alpha|u|_*$$ for all $u$, $v$ in $E$ and $\alpha\in\mathbb{R}_+$.
\\ Moreover, $$|u|_*=\|u\|_*, \forall u\in\mathcal{S},$$   $$|u|_*\leq\|u\|_*\leq C_E\|u\|, \forall u\in E,$$ and
 $$|Q[u]|_*\le C_*(1+|u|_*), \forall u\in \mathcal{S},$$ then $$\mathcal{S}\subset \overline{B_*\Big(O,(2R_*+1)e^{(C_*+1)T}\Big)}:=\overline{\Big\{u\in E \Big{|} \|u\|_*\le (2R_*+1)e^{(C_*+1)T}\Big\}},$$ for some positive constant $R_*\ge 1$.
\item [$(\mathfrak{B})$] Sub-tangent condition
\begin{equation*}
\liminf_{h\rightarrow0^+}h^{-1}\text{dist}\big(u+hQ[u],\,\mathcal{S}\big)=0,\qquad \forall\,u\in\mathcal{S}\cap B_*\Big(O,(2R_*+1)e^{(C_*+1)T}\Big)\,,
\end{equation*}

\item [$(\mathfrak{C})$] H\"{o}lder continuity condition
\begin{equation*}
\big\|Q[u] - Q[v]\big\| \leq C\|u - v\|^{\beta},\quad \beta\in(0,1), \quad \forall\,u,v\in\mathcal{S}\,,
\end{equation*}
\item [$(\mathfrak{D})$] one-side Lipschitz condition
\begin{equation*}
\big[ Q[u] - Q[v], u - v \big] \leq C\|u - v\|,\qquad \forall\,u,v\in\mathcal{S}\,,
\end{equation*}
where $$\big[ \varphi,\phi \big]: = \lim_{h\rightarrow 0^{-}}h^{-1}\big(\| \phi + h\varphi \| - \| \phi \| \big).$$
\end{itemize}
Then the equation 
\begin{equation}\label{Theorem_ODE_Eq}
\partial_t u=Q[u] \mbox{ on } [0,T]\times E,~~~~u(0)=u_0 \in \mathcal{S}\cap B_*(O,R_*)
\end{equation}
has a unique solution in $C^1((0,T),E)\cap C([0,T],\mathcal{S})$.
\end{theorem}


We shall apply Theorem \ref{Theorem:ODE} for \eqref{WeakTurbulenceGeneralized}, which reads
$$ \partial_t f = \widetilde{Q}[f], \qquad\qquad  \widetilde{Q}[f] := Q[f] - 2\nu (|k|^2+\varrho|k|^4) f,$$ 
in which $\varrho>0$.

Fix an $N>1$. We choose the Banach space $E=L^1_\frac13(\RR^d) \cap L^{1}_{N}\big(\mathbb{R}^{3}\big)$, endowed with the following norm 
$$ \| f \|_{E}: = \| f \|_{L^1_\frac13} + \|f\|_{L^1_N}.$$

We define the function $|\cdot|_*$ to be
$$|f|_*=\int_{\mathbb{R}^d}f(p)\mathcal{E}_k^\frac13 dk.$$
Set
$$\|f\|_*=\|f\|_{L^1_\frac13}.$$
By \eqref{Lemma:13Norm:E4}, it is clear that for all $f\geq0$, $f\in E$, the following inequality holds true
\begin{equation}\label{CStar}
|Q[f]|_*\le C^*\left(1+\|f\|_*\right).
\end{equation}
We then choose $C_*$ in Theorem \ref{Theorem:ODE} as $C^*$.\\ 

In addition, we take $\mathcal{S}_\varrho$ to be consisting of radial functions $f \in  L^1_\frac13(\RR^d) \cap L^{1}_{N+3}\big(\mathbb{R}^{3}\big) $ so that 
\begin{itemize}

\item[(S1)] $f\ge 0$;

\item[(S2)] $\| f\|_{L^1_\frac13} \le c_0$;

\item[(S3)] $\| f\|_{L^1_1} \le c_1$;

\item[(S4)] $\| f\|_{L^1_{N+3}} \le c_2$; 

\end{itemize}
where 
\begin{equation}\label{CONDITION:INITIALMASS}
{c}_0:=(2\mathcal{R}+1)e^{(C^*+1)T},
\end{equation}
$\mathcal{R}$, $c_1$ are some positive constant  and 
\begin{equation}\label{CONDITION:INITIALMOMENT}
{c}_{2}=\frac{3\rho_*}{2},
\end{equation}
with $\rho_*$ defined below in \eqref{RHO}. Note that from \eqref{Lemma:13Norm:E4}, $C^*$  depends on $c_1$ and $c_2$. Clearly, $\mathcal{S}_\varrho$ is a bounded, convex and closed subset of $(E,\|\cdot \|_E)$.  Moreover for all $f$ in $\mathcal{S}_\varrho$, it is straightforward that $|f|_*=\|f\|_*$. By Proposition \ref{Propo:MomentsPropa} and Remark \ref{rem-radial}, for $f_0 \in \mathcal{S}_\varrho$, solutions to \eqref{WeakTurbulenceGeneralized} are radial and remain in $\mathcal{S}_\varrho$. Thus, it suffices to verify the four conditions $(\mathfrak{A})$, $(\mathfrak{B})$, $(\mathfrak{C})$ and $(\mathfrak{D})$ of Theorem \ref{Theorem:ODE}.

\subsubsection{Condition $(\mathfrak{A})$}\label{Con2}
We choose the constant $R_*$ to be $\mathcal{R}$, then for all $u$ in $\mathcal{S}$,  $\|u\|_*\le (2{R}_*+1)e^{(C^*+1)T}$. Condition  $(\mathfrak{A})$ is satisfied.

\subsubsection{Condition $(\mathfrak{B})$}\label{Subtangent}
For the sake of simplicity, we denote $N+3$ by $N_*$.
By using Proposition \ref{Propo:C12} and recalling the definition of $\mathfrak{M}_M$, for any $g$ that makes the integrals well-defined, we have
$$\widetilde {Q}[g] \le - 2\nu\varrho \mathfrak{M}_{N_*+2}[g] +  C  \sum_{n=[N_*/2]}^{N_*-1} \Big( \mathfrak{M}_{n+\frac{7-2d}{3}}[g]   \mathfrak{M}_{N_*-n+\frac{2d-1}{3}}[g]  +  \mathfrak{M}_{n+\frac{11-2d}{3}}[g]   \mathfrak{M}_{N_*-n+\frac{2d-5}{3}}[g]\Big) .
$$
Now using Lemma \ref{lem-Holder}, with $p=1$ and $M = N_* + 1$, we get 
$$
\begin{aligned}
 \sum_{n=[N_*/2]}^{N_*-1} \Big( \mathfrak{M}_{n+\frac{7-2d}{3}}[g]   \mathfrak{M}_{N_*-n+\frac{2d-1}{3}}[g]  +  \mathfrak{M}_{n+\frac{11-2d}{3}}[g]   \mathfrak{M}_{N_*-n+\frac{2d-5}{3}}[g]\Big) 
  & \le  2\nu\varrho\mathfrak{M}_1 [g]\mathfrak{M}_{N_*+1}[g].
    \end{aligned}
$$
By assuming that $\mathfrak{M}_1[g]$ is bounded by $c_1$, we find
\begin{equation*}
\int_{\mathbb{R}^d}\widetilde {Q}[f](k)\cE_{k}^{N_*}dk \le  C\mathfrak{M}_{N_*+1}[g] - 2\nu\varrho\mathfrak{M}_{N_*+2}[g].
\end{equation*}
Now, since $C|k|^\frac43-\nu\varrho |k|^2$ is bounded for all $k$ by some positive constant $c$, we deduce that $C\mathfrak{M}_{N_*+1}[g] - \nu\varrho\mathfrak{M}_{N_*+2}[g]$ is also bounded by $C \mathfrak{M}_{N_*}[g]$. We then obtain the following estimate on $\widetilde {Q}$
\begin{equation*}
\int_{\mathbb{R}^d}\widetilde {Q}[f](k)\cE_{k}^Ndk \le  C\mathfrak{M}_{N_*}[g] - \nu\varrho\mathfrak{M}_{N_*+2}[g].
\end{equation*}
Applying again the Holder's inequality \eqref{Holder}, we end up with 
$$ \mathfrak{M}_{N_*}[g]  \le \mathfrak{M}^{\frac{2}{N_*+1}}_{1}[g] \mathfrak{M}^{\frac{N_*-1}{N_*+1}}_{N_*+2}[g] \le C\mathfrak{M}^{\frac{N_*-1}{N_*+1}}_{N_*+2}[g] .$$

Combining the above two estimates yields
\begin{equation}\label{Subtangent:E1}\begin{aligned}
&~~\int_{\mathbb{R}^d}\widetilde {Q}[f](k)\cE_{k}^{N_*}dk \leq \mathcal{P}\big[\mathfrak{M}_{N_*}[g] ]:=\ C_1\ \mathfrak{M}_{N_*}[g]  \Big( 1- C_2\mathfrak{M}^{\frac{2}{N_*-1}}_{N_*}[g]\Big)
\end{aligned}
\end{equation}
 where $C_1,C_2$ are positive constants depending on ${c}_1$. We set 
\begin{equation}\label{RHO}
\rho_* = C_2^{-\frac{N_*-1}{2}}.
\end{equation} 
Note that the function $\mathcal{P}(\cdot)$ in \eqref{Subtangent:E1} satisfies $\mathcal{P}(x)<0$ for  $0<x<\rho_*$ and $\mathcal{P}(x)>0$ for $x>\rho_*$. In addition, we may take $C_2$ in \eqref{Subtangent:E1} smaller, if needed, which allows $\rho_*$ and so $c_2$ in \eqref{CONDITION:INITIALMOMENT} to be arbitrarily large (but fixed).

Let $f$ be an arbitrary element of the set $\mathcal{S}_\varrho\cap B_*\Big(O,(2R_*+1)e^{(C_*+1)T}\Big)$. It suffices to prove the following claim: for all $\epsilon>0$, there exists $h_*$ depending on $f$ and $\epsilon$ such that 
\begin{equation}\label{claim}B_{E}(f+h\widetilde Q[f],h\epsilon)\cap\mathcal{S}_\varrho \not =\emptyset , \qquad 0<h<h_*,\end{equation}
in which $B_E(f,R)$ denotes the ball in $(E,\|\cdot \|_E)$ centered at $f$ and having radius $R$. 
For $R>0$, let $\chi_R(k)$ to be the characteristic function of the ball $B_E(0,R)$, and set  
\begin{equation}\label{def-wR} w_R:=f +h\widetilde Q[f_R] , \qquad\quad f_R(k)=\chi_R(k)f(k),\end{equation}
recalling $\widetilde{Q}[g] = Q[g] - 2\nu (|k|^2 +\varrho |k|^4) g$. We shall prove that for all $R>0$, there exists an $h_R$ so that $w_R$ belongs to $\mathcal{S}_\varrho$, for all $0<h\le h_R$. In view of \eqref{bound-Qgh}, it is clear that $w_R \in L^{1}_1 \cap L^1_{N_*}(\RR^d)$. We now check the conditions (S1)-(S3).

~\\
{\bf Condition (S1).} Note that one can write $Q[f] = Q^\mathrm{gain}[f] - Q^\mathrm{loss}[f]$, with $Q^\mathrm{gain}[f] \ge 0$ and $Q^\mathrm{loss}[f] = f Q_-[f]$. 
Since $f_R$ is compactly supported, it is clear that $Q_-[f_R]$ is bounded by a positive constant $C_f$, depending on $f,R,c_1,$ and $c_2$. Hence, 
$$\begin{aligned}
 w_R 
 &= f  + h \Big( Q[f_R]  - 2 \nu (|k|^2 +\varrho |k|^4) f_R \Big) 
 \\&\ge   f  -  h f_R \Big( C_f +  2\nu R^2 +\varrho R^4 \Big) 
 \end{aligned}$$
which is nonnegative, for sufficiently small $h$; precisely, $h \le h_R: =(C_f + 2\nu R^2+\varrho R^4)^{-1}$.

~\\
{\bf Condition (S2).} Since $$\|f\|_*<(2R_*+1)e^{(C_*+1)T},$$
and 
$$\lim_{h\to 0}\|f-w_R\|_*=0,$$
we can choose $h_*$ small enough such that 
$$\|w_R\|_*<(2R_*+1)e^{(C_*+1)T}.$$

~\\
{\bf Condition (S3).} Using Lemma \ref{Lemma:WeakFormulation} with $\varphi (k) =\cE_k$, we have  
\begin{equation}\label{est-tQQ2}
\int_{\RR^d} \widetilde Q[f_R] \cE_k \; dk \le - 2\nu \|f_R\|_{L^1_{\frac73}} \le 0.
\end{equation} 
Hence, since $f\in \mathcal{S}_\varrho$,  
$$\int_{\mathbb{R}^d}w_R\mathcal{E}_kdk=\int_{\mathbb{R}^d}(f+h\widetilde Q[f_R])\mathcal{E}_kdk \le \int_{\mathbb{R}^d}f\mathcal{E}_kdk \le c_1.$$

~\\
{\bf Condition (S4).} Now, we claim that $R$ and $h_*$ can be chosen, such that
\begin{equation}\label{condRHO}\int_{\mathbb{R}^d}w_R\mathcal{E}_k^{N_*} dk<\frac{3\rho_*}{2}\end{equation}
with $\rho_*$ defined as in \eqref{RHO}. 
In order to see this, we consider two cases. First, if 
$$\int_{\mathbb{R}^d}f\mathcal{E}_k^{N_*} dk<\frac{3\rho_*}{2},$$
we deduce from the fact  
$$\lim_{h\to 0}\int_{\mathbb{R}^d}|w_R-f|\mathcal{E}_k^{N_*}dk =\lim_{h\to 0} h\int_{\mathbb{R}^d} \widetilde Q[f_R]  \mathcal{E}_k^{N_*}dk  =0,$$
that we can choose $h_*$ small enough such that \eqref{condRHO} holds.
On the other hand, if we have $$\int_{\mathbb{R}^d}f\mathcal{E}_k^{N_*} dk=\frac{3\rho_*}{2},$$
we can then choose $R$ large enough such that
$$\int_{\mathbb{R}^d}f_R\mathcal{E}_k^{N_*} dk>{\rho_*},$$
which implies, by \eqref{Subtangent:E1} and \eqref{RHO}, that
$$\int_{\mathbb{R}^d}\widetilde {Q}[f_R]\mathcal{E}^{N_*}_k <0.$$
The estimate \eqref{condRHO} follows by definition of $w_R$.


To conclude, $w_R$ defined as in \eqref{def-wR} belongs to $\mathcal{S}_\varrho$, for $0<h\le h_R$ for sufficiently large $R$. In addition, by definition, we compute 
$$\lim_{R\to\infty}\frac{1}{h}\|w_R-f-h\widetilde Q[f]\|_E=\lim_{R\to\infty}\|\widetilde Q[f]- \widetilde Q[f_R]\|_E=0,$$
thanks to the Holder property of $\widetilde Q[f]$ with respect to $\|\cdot \|_E$. This proves that for all $\epsilon >0$, there is a large $R_\epsilon$ so that 
$w_{R_\eps}\in B_E(f+hQ[f], h\epsilon)$, for all $0<h\le h_{R_\eps}$. This proves the claim \eqref{claim}, and hence condition $(\mathfrak{B})$ is verified.  

\subsubsection{Condition $(\mathfrak{C})$}\label{Holder}
Condition $(\mathfrak{C})$ follows from Proposition \ref{Propo:HolderC12}.

\subsubsection{Condition $(\mathfrak{D})$}\label{Lipschitz}
By the Lebesgue's dominated convergence theorem, we have that
$$
\begin{aligned}
\Big[\varphi,\phi\Big] 
&= \lim_{h\rightarrow 0^{-}}h^{-1}\big(\| \phi + h\varphi \|_E - \| \phi \|_E \big)
\\
 &= \lim_{h\rightarrow 0^{-}}h^{-1} \int_{\RR^d} ( | \phi + h \varphi| - |\phi| )  (\cE_k + \cE_k^N)\; dk
\\&
\le \int_{\mathbb{R}^d}\varphi(k)\mathrm{sign}(\phi(k)) (\cE_k + \cE_k^N)dk.
\end{aligned}$$
Hence, recalling $\widetilde Q[f] = Q[f] - 2\nu ( |k|^2 + \varrho|k|^4)f$, we estimate 
$$
\begin{aligned}
\big[ \widetilde{Q}[f] - \widetilde{Q}[g], f - g \big] 
& \le 
\int_{\mathbb{R}^d}[\widetilde Q[f](k)-\widetilde Q[g](k)]\mathrm{sign}((f-g)(k))  (\cE_k^\frac13 + \cE_k^N)dk
\\
& \le 
\|Q[f]-Q[g]\|_E  -2 \nu   \| (|k|^2+\varrho|k|^4)(f-g)\|_E.
\end{aligned}$$
Using Lemma \ref{lem-QQ} and recalling $\|\cdot \|_E = \|\cdot \|_{L^1_\frac13} + \|\cdot \|_{L^1_N}$, we have
$$
\begin{aligned}
\|Q[f]-Q[g]\|_{E} 
& \le  C_{N}  \Big(\| f-g\|_{L^1_{\frac13}}+ \| f-g\|_{L^1_{N+2}} \Big) .
\end{aligned}$$
Since $C(|k|^3- \varrho |k|^4)$ is always bounded for $\varrho>0$,
we obtain 
$$
\begin{aligned}
\big[ \widetilde{Q}[f] - \widetilde{Q}[g], f - g \big]  
&\le  
  C_N \| f-g\|_{E} .\end{aligned}$$

The condition $(\mathfrak{C}$) follows. The proof of Theorem \ref{Theorem:Main} is complete for $\varrho>0$.  

\subsection{Case 2: $\varrho=0$}

Denote $f_n$ to be the unique solution to \eqref{WeakTurbulenceGeneralized} for $\varrho=\frac{1}{n}$, starting with the same initial condition $f_0$ in $\cap_{1}^\infty \mathcal{S}_n$. Proposition \ref{Propo:MomentsPropa} asserts that  $f_n$ is uniformly bounded in $L^\infty(0,\infty,L^1_{N}(\RR^d))$ for all $n$. Moreover, according to Proposition \ref{Pro:L2}, $f_n$ is uniformly bounded in $L^\infty(0,T,L^2(\RR^d))$ for all $n$. By the Dunford-Pettis theorem and Smulian's theorem, the sequence $f_n$ is equicontinuous in $t$ and it converges up to a subsequence to a nonnegative to a function $f\geq 0$ in the weak $L^1$ sense. Recalling from \eqref{Q-L1bound} that $Q[f]$ is Lipschitz from $L^1_{\frac13} \cap L^1_{N+2}$ to $L^1_N$, and $f_n$ converges weakly to $f$ in $L^1_s(\RR^d)$ for all $s\in [1,N+3]$. This implies that $Q[f_n]$ also converges to $Q[f]$ in the the weak $L^1$ sense. As a consequence, $f$ is a solution of \eqref{WeakTurbulenceInitial}.

\appendix
 \section{Appendix: Proof of Theorem \ref{Theorem:ODE}}\label{Appendix}
We recall below the proof of Theorem  \ref{Theorem:ODE}, which is Theorem 1.3 of \cite{SofferBinh1}, for the sake of completeness. The proof is divided into four parts. 
{\\\\\bf Part 1:}
Fix a element $v$ of $\mathcal{S}$, due to the H\"older continuity property of $Q[u]$, we have $$\|Q[u]\|\le \|Q[v]\|+ C\|u-v\|^\beta,~~~~~\forall u\in\mathcal{S}.$$ 
According to our assumption, $\mathcal{S}$ is bounded by a constant $C_S$. We deduce from the above inequality that 
$$\|Q[u]\|\le \|Q[v]\|+ C\left(\|u\|+\|v\|\right)^\beta\le \|Q[v]\|+ C\left(C_S+\|v\|\right)^\beta=:C_Q, \ \forall u \in \mathcal{S}.$$ 
For an element $u$ be in $\mathcal{S}$, there exists $\xi_u>0$ such that for $0<\xi<\xi_u$, $u+\xi Q[u]\in{\mathcal{S}}$, which implies $$B(u+\xi Q[u],\delta)\cap {\mathcal{S}}\backslash\{u+\xi Q[u]\}\ne {\O},$$ for $\delta$ small enough. Choose $\epsilon=2C((C_Q+1)\xi)^\beta$, then $\|Q[u]-Q[v]\|\leq \frac{\epsilon}{2}$ if $\|u-v\|\leq (C_Q+1)\xi$, by the H\"older continuity of $Q$. Let $z\in B\left(u+\xi Q[u],\frac{\epsilon \xi}{2}\right)\cap {\mathcal{S}}\backslash\{u+\xi Q[u]\} $ and define
$$t\mapsto \vartheta(t)=u+\frac{t(z-u)}{\xi},~~~~t\in[0,\xi].$$
Since  $\mathcal{S}$ is convex, $\vartheta$ maps $[0,\xi]$ into $\mathcal{S}$. 
It is straightforward that $$\|\vartheta(t)-u\|\leq \xi\|Q[u]\|+\frac{\epsilon \xi}{2}<(C_Q+1)\xi,$$
which implies $$\|Q({\vartheta}(t))-Q[u]\|\leq \frac{\epsilon}{2},~~\forall t\in[0,\xi].$$
The above inequality and the fact that $$\|\dot{\vartheta}(t)-Q[u]\|\leq \frac{\epsilon}{2},$$
leads to 
\begin{equation}\label{Theorem:ODE:E1}
\|\dot{\vartheta}(t)-Q({\vartheta}(t))\|\leq \epsilon,~~\forall t\in[0,\xi].
\end{equation}
 {\\\bf Part 2:} Let $\vartheta$ be a solution to \eqref{Theorem:ODE:E1} on $[0,\tau]$. Inequality \eqref{Theorem:ODE:E1} leads to
$$\left|\frac{\vartheta(\tau)-\vartheta(0)}{\tau}-Q(\vartheta(0))\right|_*\le C_E \epsilon,$$
which yields
 $$|\vartheta(\tau)|_*\leq |\vartheta(0)|_*+\tau C_*(|\vartheta(0)|_*+1)+\tau C_E \epsilon.$$
Since we can assume that $C_E \epsilon<1$, we obtain
\begin{equation}\label{Theorem:ODE:E4}
|\bar\vartheta(\tau)|_*\le (|\bar\vartheta(0)|_*+1)e^{(C_*+1)\tau}-1.\end{equation}
 Using the procedure of Part 1, we assume that $\bar\vartheta$ can be extended to the interval $[\tau,\tau+\tau']$.
 \\ The same arguments that lead to \eqref{Theorem:ODE:E4} imply
  $$|\bar\vartheta(\tau+\tau')|_*\le \left( (|\bar\vartheta(\tau)|_*+1)e^{(C_*+1)\tau'}-1\right).$$
Combining the above inequality with \eqref{Theorem:ODE:E4} yields
\begin{equation}\label{Theorem:ODE:E5}
\begin{aligned}
\|\bar\vartheta(\tau+\tau')\|_*=|\bar\vartheta(\tau+\tau')|_*\
\le&~~\left(|\bar\vartheta(0)|_*+1\right)\left(e^{(C_*+1)(\tau+\tau')}-1\right)\\
\le&~~ (2R_*+1)e^{(C_*+1)(\tau+\tau')},
\end{aligned}
\end{equation}
where the last inequality follows from the fact that $R_*\ge 1$.
{\\\\\bf Part 3:} From Part 1, there exists a solution $\vartheta$ to the equation \eqref{Theorem:ODE:E1} on an interval $[0,h]$. Now, we have the following procedure.
\begin{itemize}
\item {\it Step 1:}  Suppose that  we can construct the solution $\vartheta$ of \eqref{Theorem:ODE:E1} on $[0,\tau]$ $(\tau<T)$. Since $\vartheta(\tau)\in\mathcal{S}$, by the same process as in Part 1 and by \eqref{Theorem:ODE:E4} and \eqref{Theorem:ODE:E5}, the solution $\vartheta$ could be extended to  $[\tau,\tau+h_\tau]$ where $\tau+h_\tau\le T, h_\tau\le \tau$. 
\item {\it Step 2:} Suppose that we can construct the solution $\vartheta$ of \eqref{Theorem:ODE:E1}  on a series of intervals $[0,\tau_1]$, $[\tau_1,\tau_2]$, $\cdots$, $[\tau_n,\tau_{n+1}]$, $\cdots$. Observe that the increasing sequence $\{\tau_n\}$ is bounded by $T$, the sequence has a limit, defined by $\tau.$
Recall that $Q({\vartheta})$ is bounded by $C_Q$ on $[\tau_n,\tau_{n+1}]$ for all $n\in\mathbb{N},$ then $\dot{\vartheta}$ is bounded by $\epsilon+C_Q$ on $[0,\tau)$. As a consequence  $\vartheta(\tau)$ can be defined as $$\vartheta(\tau)=\lim_{n\to\infty}\vartheta(\tau_n), \dot{\vartheta}(\tau)=\lim_{n\to\infty}\dot{\vartheta}(\tau_n),$$
which, together with the fact that $\mathcal{S}$ is closed, implies that $\vartheta$ is a solution of \eqref{Theorem:ODE:E1} on $[0,\tau]$. 
\end{itemize}
By Step 2, if
 the solution $\vartheta$ can be defined on $[0,T_0)$, $T_0<T$, it could be extended to $[0,T_0]$. Now, we suppose that $[0,T_0]$ is the maximal  closed interval that $\vartheta$ could be defined, by Step 1, $\vartheta$ could be extended to a larger interval $[T_0,T_0+T_h]$, which means that $T=T_0$ and $\vartheta$ is defined on the whole interval $[0,T]$.
{\\\\\bf Part 4:} Finally, let us consider a sequence of solution $\{u^\epsilon\}$ to \eqref{Theorem:ODE:E1} on $[0,T]$. We will prove that this is a Cauchy sequence. 
Let $\{u^\epsilon\}$ and $\{v^\epsilon\}$ be two sequences of solutions to \eqref{Theorem:ODE:E1} on $[0,T]$. We note that $u^\epsilon$ and $v^\epsilon$ are affine functions on $[0,T]$. Moreover by the one-side Lipschitz condition
\begin{eqnarray*}
\frac{d}{dt} \|u^\epsilon(t)-v^\epsilon(t)\|&=&\Big[u^\epsilon(t)-v^\epsilon(t),\dot{u}^\epsilon(t)-\dot{v}^\epsilon(t)\Big]\\
&\le& \Big[u^\epsilon(t)-v^\epsilon(t),Q[u^\epsilon(t)]-Q[v^\epsilon(t)]\Big]+2\epsilon\\
&\le& C\|u^\epsilon(t)-v^\epsilon(t)\|+2\epsilon, 
\end{eqnarray*}
for a.e. $t\in[0,T]$, which leads to
$$\|u^\epsilon(t)-v^\epsilon(t)\|\le 2\epsilon\frac{e^{LT}}{L}.$$
By letting $\epsilon$ tend to $0$, $u^\epsilon\to u$ uniformly on $[0,T]$. It is straightforward that $u$ is a solution to \eqref{Theorem_ODE_Eq}.

~\\
{\bf Acknowledgements.} 
TN's research was supported in part
by the NSF under grant DMS-1405728. M.-B. Tran has been supported by NSF Grant RNMS (Ki-Net) 1107291, ERC Advanced Grant DYCON. The authors would like to thank Professor Yves Pomeau for his constructive  comments on the previous version of the paper. They would also like to express their gratitude to Professor Sergey Nazarenko for explaining to them the difference between the two works \cite{pushkarev1996turbulence} and \cite{zakharov1968stability}, which led to a major improvement of the earlier manuscript. They are also grateful to Professor Colm Connaughton and Professor Leslie M. Smith for  the discussions.

\bibliographystyle{plain}
\bibliography{QuantumBoltzmann}

\end{document}